\documentclass[reqno]{amsart}
\usepackage{amssymb}
\usepackage{amsmath}
\usepackage{amsfonts}

\newtheorem{theorem}{Theorem}[section]
\theoremstyle{plain}

\newtheorem{lemma}{Lemma}[section]

\numberwithin{equation}{section}

\begin{document}
\title[Long-time dynamics]{Long-time dynamics of the strongly damped
semilinear plate equation in $%
\mathbb{R}
^{n}$}
\author{Azer Khanmamedov \ }
\address{{\small Department of Mathematics,} {\small Faculty of Science,
Hacettepe University, Beytepe 06800}, {\small Ankara, Turkey}}
\email{azer@hacettepe.edu.tr}
\author{\ Sema Yayla}
\address{{\small Department of Mathematics,} {\small Faculty of Science,
Hacettepe University, Beytepe 06800}, {\small Ankara, Turkey}}
\email{semasimsek@hacettepe.edu.tr}
\subjclass[2000]{ 35B41, 35G20, 74K20}
\keywords{wave equation, plate equation, global attractor}

\begin{abstract}
We investigate the initial-value problem for the semilinear plate equation
containing localized strong damping, localized weak damping and nonlocal
nonlinearity. We prove that if nonnegative damping coefficients are strictly
positive almost everywhere in the exterior of some ball and the sum of these
coefficients is positive a.e. in $%
\mathbb{R}
^{n}$, then the semigroup generated by the considered problem possesses a
global attractor in $H^{2}\left(
\mathbb{R}
^{n}\right) \times L^{2}\left(
\mathbb{R}
^{n}\right) $. We also establish boundedness of this attractor in $%
H^{3}\left(
\mathbb{R}
^{n}\right) \times H^{2}\left(
\mathbb{R}
^{n}\right) $.
\end{abstract}

\maketitle

\section{Introduction}

\bigskip In this paper, our main purpose is to study the long-time dynamics
(in terms of attractors) of the plate equation\ \ \ \ \ \ \ \ \ \ \ \ \ \ \
\ \ \ \ \ \ \ \ \ \ \ \ \ \ \ \ \ \ \ \ \
\begin{equation}
u_{tt}+\gamma \Delta ^{2}u-{div}\left( \beta \left( x\right) \nabla
u_{t}\right) +\alpha (x)u_{t}+\lambda u-f(\left\Vert \nabla u\left( t\right)
\right\Vert _{L^{2}\left(
\mathbb{R}
^{n}\right) })\Delta u+g\left( u\right) =h\left( x\right) \text{, \ }%
(t,x)\in
\mathbb{R}
^{+}\times
\mathbb{R}
^{n}\text{,}  \tag{1.1}
\end{equation}%
with initial data%
\begin{equation}
u(0,x)=u_{0}(x)\text{, \ }u_{t}(0,x)=u_{1}(x)\text{, \ \ }x\in
\mathbb{R}
^{n}\text{, \ \ \ \ \ \ \ \ \ \ \ \ \ \ \ \ \ \ \ \ \ \ \ \ \ \ \ \ \ \ }
\tag{1.2}
\end{equation}%
where $\gamma >0$, $\lambda >0$, $h\in L^{2}\left(
\mathbb{R}
^{n}\right) $ and the functions $\alpha \left( \cdot \right) ,$ $\beta
\left( \cdot \right) ,$ $f\left( \cdot \right) $ and $g\left( \cdot \right) $
satisfy the following conditions:%
\begin{equation}
\alpha ,\text{ }\beta \in L^{\infty }(%
\mathbb{R}
^{n})\text{, }\alpha (\cdot )\geq 0,\text{ }\beta \left( \cdot \right) \geq 0%
\text{\ \ a.e. in }%
\mathbb{R}
^{n}\text{,}  \tag{1.3}
\end{equation}%
\begin{equation}
\alpha (\cdot )\geq \alpha _{0}>0\text{ and }\beta (\cdot )\geq \beta _{0}>0%
\text{ a.e. in }\left\{ x\in
\mathbb{R}
^{n}:\left\vert x\right\vert \geq r_{0}\right\} \text{, for some }r_{0}>0%
\text{,}  \tag{1.4}
\end{equation}%
\begin{equation}
\alpha (\cdot )+\beta \left( \cdot \right) >0\text{ \ a.e. in }%
\mathbb{R}
^{n}\text{,}  \tag{1.5}
\end{equation}%
\begin{equation}
f\in C^{1}(%
\mathbb{R}
^{+}),\text{ }f\left( z\right) \geq 0\text{, for all }z\in
\mathbb{R}
^{+},  \tag{1.6}
\end{equation}%
\begin{equation}
g\in C^{1}(%
\mathbb{R}
)\text{, }\left\vert g^{\prime }(s)\right\vert \leq C\left( 1+\left\vert
s\right\vert ^{p-1}\right) \text{, }p\geq 1\text{, }(n-4)p\leq n\text{,}
\tag{1.7}
\end{equation}%
\begin{equation}
\text{ }g(s)s\geq 0\text{, for every }s\in
\mathbb{R}
\text{.}  \tag{1.8}
\end{equation}

The problem (1.1)-(1.2) can be reduced to the following Cauchy problem for
the first order abstract differential equation in the space $H^{2}\left(
\mathbb{R}
^{n}\right) \times $\ $L^{2}\left(
\mathbb{R}
^{n}\right) $:%
\begin{equation*}
\left\{
\begin{array}{c}
\frac{d}{dt}\theta (t)=A\theta (t)+\mathcal{F(}\theta (t)), \\
\theta (0)=\theta _{0},%
\end{array}%
\right.
\end{equation*}%
\ where $\theta (t)=(u(t),u_{t}(t))$,\ $\theta _{0}=(u_{0},u_{1})$, $A(u,$ $%
v)=(v,$ $-\gamma \Delta ^{{\small 2}}u+{div}\left( \beta \left( \cdot
\right) \nabla v\right) -\alpha (\cdot )v-\lambda u)$, $D(A)=\left\{
(u,v)\in H^{3}\left(
\mathbb{R}
^{n}\right) \times \ H^{2}\left(
\mathbb{R}
^{n}\right) :\text{ }\gamma \Delta ^{{\small 2}}u-{div}\left( \beta \left(
\cdot \right) \nabla v\right) \in L^{2}\left(
\mathbb{R}
^{n}\right) \right\} $ and $\mathcal{F}(u,v)=(0,$ $f(\left\Vert \nabla
u\right\Vert _{L^{2}\left(
\mathbb{R}
^{n}\right) })\Delta u$\newline
$-g\left( u\right) +h)$. Defining suitable equivalent norm in $H^{2}\left(
\mathbb{R}
^{n}\right) \times $\ $L^{2}\left(
\mathbb{R}
^{n}\right) $, it is easy to see that the operator $A$, thanks to (1.3), is
maximal dissipative in $H^{2}\left(
\mathbb{R}
^{n}\right) \times $\ $L^{2}\left(
\mathbb{R}
^{n}\right) $ and consequently, due to Lumer-Phillips Theorem (see \cite[%
Theorem 4.3]{1}), it generates a linear continuous semigroup $\left\{
e^{tA}\right\} _{t\geq 0}$. Also, by (1.6)-(1.7), we find that the nonlinear
operator $\mathcal{F}:H^{2}\left(
\mathbb{R}
^{n}\right) \times $\ $L^{2}\left(
\mathbb{R}
^{n}\right) \rightarrow H^{2}\left(
\mathbb{R}
^{n}\right) \times $\ $L^{2}\left(
\mathbb{R}
^{n}\right) $ is Lipschitz continuous on bounded subsets of $H^{2}\left(
\mathbb{R}
^{n}\right) \times $\ $L^{2}\left(
\mathbb{R}
^{n}\right) $. So, applying semigroup theory (see, for example \cite[p. 56-58%
]{2}), and taking advantage of energy estimates, we have the following
well-posedness result.

\begin{theorem}
Assume that the conditions (1.3), (1.6), (1.7) and (1.8) hold. Then, for
every $\left( u_{0},u_{1}\right) \in H^{2}\left(
\mathbb{R}
^{n}\right) \times L^{2}\left(
\mathbb{R}
^{n}\right) $, the problem (1.1)-(1.2) has a unique weak solution $u\in
C\left( [0,\infty );H^{2}\left(
\mathbb{R}
^{n}\right) \right) $\newline
$\cap $ $C^{1}\left( [0,\infty );L^{2}\left(
\mathbb{R}
^{n}\right) \right) $, which depends continuously on the initial data and
satisfies the energy equality
\begin{equation*}
E\left( u\left( t\right) \right) +\int\limits_{%
\mathbb{R}
^{n}}G\left( u\left( t,x\right) \right) dx+\frac{1}{2}F\left( \left\Vert
\nabla u\left( t\right) \right\Vert _{L^{2}\left(
\mathbb{R}
^{n}\right) }^{2}\right) -\int\limits_{%
\mathbb{R}
^{n}}h\left( x\right) u\left( t,x\right) dx
\end{equation*}%
\begin{equation*}
+\int\limits_{s}^{t}\int\limits_{%
\mathbb{R}
^{n}}\alpha \left( x\right) \left\vert u_{t}\left( \tau ,x\right)
\right\vert ^{2}dxd\tau +\int\limits_{s}^{t}\int\limits_{%
\mathbb{R}
^{n}}\beta \left( x\right) \left\vert \nabla u_{t}\left( \tau ,x\right)
\right\vert ^{2}dxd\tau
\end{equation*}%
\begin{equation}
\text{ }=E\left( u\left( s\right) \right) +\int\limits_{%
\mathbb{R}
^{n}}G\left( u\left( s,x\right) \right) dx+\frac{1}{2}F\left( \left\Vert
\nabla u\left( s\right) \right\Vert _{L^{2}\left(
\mathbb{R}
^{n}\right) }^{2}\right) -\int\limits_{%
\mathbb{R}
^{n}}h\left( x\right) u\left( s,x\right) dx\text{, \ }\forall t\geq s\geq 0%
\text{,}  \tag{1.9}
\end{equation}%
where $F\left( z\right) =\int\limits_{0}^{z}f\left( \sqrt{s}\right) ds$ for
all $z\in
\mathbb{R}
^{+},$ $G\left( z\right) =\int\limits_{0}^{z}g\left( s\right) ds$ for all $%
z\in
\mathbb{R}
$ and $E\left( u\left( t\right) \right) =\frac{1}{2}\int\limits_{%
\mathbb{R}
^{n}}(\left\vert u_{t}\left( t,x\right) \right\vert ^{2}+\gamma \left\vert
{\small \Delta u}\left( t,x\right) \right\vert ^{2}+\lambda \left\vert
{\small u}\left( t,x\right) \right\vert ^{2})dx$. Moreover, if $\left(
u_{0},u_{1}\right) \in D(A)$, then $u(t,x)$ is a strong solution satisfying $%
(u,u_{t})\in C\left( [0,\infty );D(A)\right) \cap $ $C^{1}[0,\infty
);H^{2}\left(
\mathbb{R}
^{n}\right) \times L^{2}\left(
\mathbb{R}
^{n}\right) $.
\end{theorem}

Thus, due to Theorem 1.1, the problem (1.1)-(1.2) generates a strongly
continuous semigroup $\left\{ S\left( t\right) \right\} _{t\geq 0}$ in $%
H^{2}\left(
\mathbb{R}
^{n}\right) \times L^{2}\left(
\mathbb{R}
^{n}\right) $ by the formula $\left( u\left( t\right) ,u_{t}\left( t\right)
\right) =S\left( t\right) (u_{0},u_{1})$, where $u\left( t,x\right) $ is a
weak solution of (1.1)-(1.2) with the initial data $\left(
u_{0},u_{1}\right) $.

Attractors for hyperbolic and hyperbolic like equations in unbounded domains
have been extensively studied by many authors over the last few decades. To
the best of our knowledge, the first works in this area were done by
Feireisl in \cite{3} and \cite{4}, for the wave equations with the weak
damping (the case $\gamma =0$, $\beta \equiv 0$ and $f\equiv 1$ in (1.1)) .
In those articles the author, by using the finite speed propagation property
of the wave equations, established the existence of the global attractors in
$H^{1}\left(
\mathbb{R}
^{n}\right) \times L^{2}\left(
\mathbb{R}
^{n}\right) $. The global attractors for the wave equations involving strong
damping in the form $-\Delta u_{t}$, besides weak damping, were investigated
in \cite{5} and \cite{6}, where the authors, by using splitting method,
proved the existence of the global attractors in $H^{1}\left(
\mathbb{R}
^{n}\right) \times L^{2}\left(
\mathbb{R}
^{n}\right) $, under different conditions on the nonlinearities. Recently,
in \cite{7}, the results of \cite{5} and \cite{6} have been improved for the
wave equation involving additional nonlocal nonlinear term in the form $%
-(a+b\left\Vert \nabla u\left( t\right) \right\Vert _{L^{2}\left(
\mathbb{R}
^{n}\right) }^{2})\Delta u$ ($a\geq 0,$ $b>0)$. For the plate equation with
only weak damping and local nonlinearity (the case $\gamma =1$, $\beta
\equiv 0$ and $f\equiv 0$ in (1.1)), attractors were investigated in \cite{8}
and \cite{9}, where the author, inspired by the methods of \cite{10} and
\cite{11}, proved the existence, regularity and finite dimensionality of the
global attractors in $H^{2}\left(
\mathbb{R}
^{n}\right) \times L^{2}\left(
\mathbb{R}
^{n}\right) $. The situation becomes more difficult when the equation
contains localized damping terms and nonlocal nonlinearities. Recently, in
\cite{12} and \cite{13}, the plate equation with localized weak damping (the
case $\beta \equiv 0$ in (1.1)) and involving nonlocal nonlinearities as $%
-f(\left\Vert \nabla u\right\Vert _{L^{2}\left(
\mathbb{R}
^{n}\right) })\Delta u$ and $\ f(\left\Vert u\right\Vert _{L^{p}\left(
\mathbb{R}
^{n}\right) })\left\vert u\right\vert ^{p-2}u$ have been considered. In
these articles, the existence of global attractors has been proved when the
coefficient $\alpha (\cdot )$ of the weak damping term is strictly positive
(see \cite{12}) or, in addition to (1.3), is positive (see \cite{13}) almost
everywhere in $%
\mathbb{R}
^{n}$. However, in the case when $\alpha (\cdot )$ vanishes in a set of \
positive measure, the existence of the global attractor for (1.1) with $%
\beta \equiv 0$ remained as an open question (see \cite[Remark 1.2]{12}). On
the other hand, in the case when $\alpha \equiv 0$ and even $\beta \equiv 1$%
, the semigroup $\left\{ S\left( t\right) \right\} _{t\geq 0}$ generated by
(1.1)-(1.2) does not possess a global attractor in $H^{2}\left(
\mathbb{R}
^{n}\right) \times L^{2}\left(
\mathbb{R}
^{n}\right) $. Indeed, if $\left\{ S\left( t\right) \right\} _{t\geq 0}$
possesses a global attractor, then the linear semigroup $\left\{
e^{tA}\right\} _{t\geq 0}$ decay exponentially in the real and consequently,
complex space $H^{2}\left(
\mathbb{R}
^{n}\right) \times L^{2}\left(
\mathbb{R}
^{n}\right) $, which, due to Hille-Yosida Theorem (see \cite[Remark 5.4]{1}%
), implies necessary condition $i%
\mathbb{R}
\subset \rho (A)$. This condition is equivalent to the solvability of the
equation $(i\mu I-A)(u,v)=(y,z)$ in $H^{2}\left(
\mathbb{R}
^{n}\right) \times L^{2}\left(
\mathbb{R}
^{n}\right) $, for every\ $(y,z)$ in $H^{2}\left(
\mathbb{R}
^{n}\right) \times L^{2}\left(
\mathbb{R}
^{n}\right) $ and $\mu \in
\mathbb{R}
$. Choosing $\mu =\sqrt{\lambda }$ and $y=0$, we have $\ v=i\sqrt{\lambda }u$
and $\Delta (\Delta u-iu)=z$. If the last equation for every $z\in
L^{2}\left(
\mathbb{R}
^{n}\right) $ has a solution $u\in $ $H^{3}\left(
\mathbb{R}
^{n}\right) $, then denoting $\varphi =\Delta u-iu$, we can say that the
equation $\Delta \varphi =z$ has a solution in $H^{1}\left(
\mathbb{R}
^{n}\right) $, for every $z\in L^{2}\left(
\mathbb{R}
^{n}\right) $. However, the last equation, as shown in \cite{6}, is not
solvable in $H^{1}\left(
\mathbb{R}
^{n}\right) $ for some $z\in L^{2}\left(
\mathbb{R}
^{n}\right) $. Hence, the necessary condition $i%
\mathbb{R}
\subset \rho (A)$ does not hold. Thus, in the case when $\alpha \equiv 0$
and $\beta \equiv 1$, the problem (1.1)-(1.2) does not have a global
attractor, and in the case when $\beta \equiv 0$ and $\alpha (\cdot )$
vanishes in a set of positive measure, the existence of the global for
(1.1)-(1.2) is an open question.

In this paper, we impose conditions (1.3)-(1.5) on damping coefficients $%
\alpha (\cdot )$ and $\beta (\cdot )$, which, unlike the conditions imposed
in the previous articles dealing with the wave and plate equations involving
strong damping and/or nonlocal nonlinearities, allow both of them to be
vanished in the sets of positive measure such that in these sets the strong
damping and weak damping complete each other. Thus, our main result is as
follows:

\begin{theorem}
Under the conditions (1.3)-(1.8) the semigroup $\left\{ S\left( t\right)
\right\} _{t\geq 0}$ generated by the problem (1.1)-(1.2) possesses a global
attractor $\mathcal{A}$ in $H^{2}\left(
\mathbb{R}
^{n}\right) \times L^{2}\left(
\mathbb{R}
^{n}\right) $ and $\mathcal{A=M}^{u}\left( \mathcal{N}\right) $. Here $%
\mathcal{M}^{u}\left( \mathcal{N}\right) $ is unstable manifold emanating
from the set of stationary points $\mathcal{N}$ (for definition, see \cite[%
359]{14}). Moreover, the global attractor $\mathcal{A}$ is bounded in $%
H^{3}\left(
\mathbb{R}
^{n}\right) \times H^{2}\left(
\mathbb{R}
^{n}\right) $.
\end{theorem}

The plan of the paper is as follows: In the next section, after the proof of
two auxiliary lemmas, we establish asymptotic compactness of $\left\{
S\left( t\right) \right\} _{t\geq 0}$ in the interior domain. Then, we prove
Lemma 2.3, which plays a key role for the tail estimate, and thereby we show
that the solutions of (1.1)-(1.2) are uniformly (with respect to the initial
data) small at infinity for large time. This fact, together with asymptotic
compactness in the interior domain, yields asymptotic compactness of $%
\left\{ S\left( t\right) \right\} _{t\geq 0}$ in the whole space, and\ by
applying the abstract result on the gradient systems, we establish the
existence of the global attractor (see Theorem 2.3). In Section 3, by using
the invariance of the global attractor, we show that it has an additional
regularity.

\section{Existence of the global attractor}

We begin with the following lemmas:

\begin{lemma}
Assume that the condition (1.6) holds. Also, assume that\ the sequence $%
\left\{ v_{m}\right\} _{m=1}^{\infty }$ is weakly star convergent in $%
L^{\infty }\left( 0,\infty ;H^{2}\left(
\mathbb{R}
^{n}\right) \right) $, the sequence $\left\{ v_{mt}\right\} _{m=1}^{\infty }$
is bounded in $L^{\infty }\left( 0,\infty ;L^{2}\left(
\mathbb{R}
^{n}\right) \right) $ and the sequence $\left\{ \left\Vert \nabla
v_{m}\left( t\right) \right\Vert _{L^{2}\left(
\mathbb{R}
^{n}\right) }\right\} _{m=1}^{\infty }$ is convergent, for all $t\geq 0$.
Then, for every $r>0$ and $\phi \in C_{0}^{1}(B\left( 0,r\right) )$%
\begin{equation*}
\underset{m\rightarrow \infty }{\lim }\text{ }\underset{l\rightarrow \infty }%
{\lim \sup }\left\vert \int\limits_{0}^{t}\int\limits_{B\left( 0,r\right)
}\tau \left( f(\left\Vert \nabla v_{m}\left( \tau \right) \right\Vert
_{L^{2}\left(
\mathbb{R}
^{n}\right) })\Delta v_{m}(\tau ,x)-f(\left\Vert \nabla v_{l}\left( \tau
\right) \right\Vert _{L^{2}\left(
\mathbb{R}
^{n}\right) })\Delta v_{l}(\tau ,x)\right) \right.
\end{equation*}%
\begin{equation*}
\left. \times \phi (x)\left( v_{mt}\left( \tau ,x\right) -v_{lt}\left( \tau
,x\right) \right) dxd\tau \right\vert =0\text{, \ }\forall t\geq 0\text{,}
\end{equation*}%
where $B\left( 0,r\right) =\left\{ x\in
\mathbb{R}
^{n}:\left\vert x\right\vert <r\right\} $.
\end{lemma}

\begin{proof}
Firstly, we have%
\begin{equation*}
\left\vert \int\limits_{0}^{t}\int\limits_{B\left( 0,r\right) }\tau \left(
f(\left\Vert \nabla v_{m}\left( \tau \right) \right\Vert _{L^{2}\left(
\mathbb{R}
^{n}\right) })\Delta v_{m}(\tau ,x)-f(\left\Vert \nabla v_{l}\left( \tau
\right) \right\Vert _{L^{2}\left(
\mathbb{R}
^{n}\right) })\Delta v_{l}(\tau ,x)\right) \right.
\end{equation*}%
\begin{equation*}
\left. \times \phi (x)\left( v_{mt}\left( \tau ,x\right) -v_{lt}\left( \tau
,x\right) \right) dxd\tau \right\vert
\end{equation*}%
\begin{equation}
\leq \frac{1}{2}\left\vert \int\limits_{0}^{t}\tau f\left( \left\Vert \nabla
v_{l}\left( \tau \right) \right\Vert _{L^{2}\left(
\mathbb{R}
^{n}\right) }\right) \frac{d}{d\tau }\int\limits_{B\left( 0,r\right) }\phi
(x)\left\vert \nabla v_{m}(\tau ,x)-\nabla v_{l}(\tau ,x)\right\vert
^{2}dxd\tau \right\vert +\left\vert K_{r}^{m,l}\left( t\right) \right\vert
\text{,}  \tag{2.1}
\end{equation}%
where $K_{r}^{m,l}\left( t\right) =$ $\int\limits_{0}^{t}\tau \left(
f(\left\Vert \nabla v_{m}\left( \tau \right) \right\Vert _{L^{2}\left(
\mathbb{R}
^{n}\right) })-f(\left\Vert \nabla v_{l}\left( \tau \right) \right\Vert
_{L^{2}\left(
\mathbb{R}
^{n}\right) })\right) \int\limits_{B\left( 0,r\right) }$ $\phi (x)\Delta
v_{m}(\tau ,x)$\newline
$\times \left( v_{mt}\left( \tau ,x\right) -v_{lt}\left( t,x\right) \right)
dxd\tau -$ $\int\limits_{0}^{t}\int\limits_{B\left( 0,r\right) }\tau f\left(
\left\Vert \nabla v_{l}\left( \tau \right) \right\Vert _{L^{2}\left(
\mathbb{R}
^{n}\right) }\right) \nabla \phi (x)\cdot \nabla \left( v_{m}(\tau
,x)-v_{l}(\tau ,x)\right) $\newline
$\times \left( v_{mt}\left( \tau ,x\right) -v_{lt}\left( \tau ,x\right)
\right) dxd\tau $. Applying \cite[Corollary 4]{15}, we have that the
sequence $\left\{ v_{m}\right\} _{m=1}^{\infty }$ $\ $is relatively compact
in $C\left( \left[ 0,T\right] ;H^{2-\varepsilon }\left( B\left( 0,r\right)
\right) \right) $, for every $\varepsilon >0$, $T>0$ and $r>0$. So,%
\begin{equation}
v_{m}\rightarrow v\text{ strongly in }C\left( \left[ 0,T\right]
;H^{2-\varepsilon }\left( B\left( 0,r\right) \right) \right) ,  \tag{2.2}
\end{equation}%
for some $v\in $ $C\left( \left[ 0,T\right] ;H^{2-\varepsilon }\left(
B\left( 0,r\right) \right) \right) $. Hence, we find
\begin{equation}
\underset{m\rightarrow \infty }{\text{ }\lim \text{ }}\underset{l\rightarrow
\infty }{\lim \sup }\left\vert K_{r}^{m,l}\left( t\right) \right\vert =0%
\text{, \ \ }\forall t\geq 0\text{.}  \tag{2.3}
\end{equation}%
Now, denoting $f_{\varepsilon }\left( u\right) $ $=\left\{
\begin{array}{c}
f\left( u\right) ,\text{ }u\geq \varepsilon \\
f\left( \varepsilon \right) ,\text{ }0\leq u<\varepsilon%
\end{array}%
\right. $ for $\varepsilon >0$, we get%
\begin{equation*}
\left\vert f\left( \left\Vert \nabla v_{l}\left( \tau \right) \right\Vert
_{L^{2}\left(
\mathbb{R}
^{n}\right) }\right) -f_{\varepsilon }\left( \left\Vert \nabla v_{l}\left(
\tau \right) \right\Vert _{L^{2}\left(
\mathbb{R}
^{n}\right) }\right) \right\vert \leq \max_{0\leq s_{1},s_{2}\leq
\varepsilon }\left\vert f\left( s_{1}\right) -f\left( s_{2}\right)
\right\vert ,
\end{equation*}%
and then, for the first term on the right hand side of (2.1), we obtain%
\begin{equation*}
\left\vert \int\limits_{0}^{t}\tau f\left( \left\Vert \nabla v_{l}\left(
\tau \right) \right\Vert _{L^{2}\left(
\mathbb{R}
^{n}\right) }\right) \frac{d}{d\tau }\int\limits_{B\left( 0,r\right) }\phi
(x)\left\vert \nabla v_{m}(\tau ,x)-\nabla v_{l}(\tau ,x)\right\vert
^{2}dxd\tau \right\vert
\end{equation*}%
\begin{equation*}
\leq \left\vert \int\limits_{0}^{t}\tau f_{\varepsilon }\left( \left\Vert
\nabla v_{l}\left( \tau \right) \right\Vert _{L^{2}\left(
\mathbb{R}
^{n}\right) }\right) \frac{d}{d\tau }\int\limits_{B\left( 0,r\right) }\phi
(x)\left\vert \nabla v_{m}(\tau ,x)-\nabla v_{l}(\tau ,x)\right\vert
^{2}dxd\tau \right\vert
\end{equation*}%
\begin{equation}
+c_{1}\text{ }t^{2}\max_{0\leq s_{1},s_{2}\leq \varepsilon }\left\vert
f\left( s_{1}\right) -f\left( s_{2}\right) \right\vert \text{, \ }\forall
t\geq 0\text{.}  \tag{2.4}
\end{equation}%
Let us estimate the first term on the right hand side of (2.4). By using
integration by parts, we have%
\begin{equation*}
\int\limits_{0}^{t}\tau f_{\varepsilon }\left( \left\Vert \nabla v_{l}\left(
\tau \right) \right\Vert _{L^{2}\left(
\mathbb{R}
^{n}\right) }\right) \frac{d}{d\tau }\int\limits_{B\left( 0,r\right) }\phi
(x)\left\vert \nabla v_{m}(\tau ,x)-\nabla v_{l}(\tau ,x)\right\vert
^{2}dxd\tau
\end{equation*}%
\begin{equation*}
=tf_{\varepsilon }\left( \left\Vert \nabla v_{l}\left( t\right) \right\Vert
_{L^{2}\left(
\mathbb{R}
^{n}\right) }\right) \int\limits_{B\left( 0,r\right) }\phi (x)\left\vert
\nabla v_{m}(\tau ,x)-\nabla v_{l}(\tau ,x)\right\vert ^{2}dx
\end{equation*}%
\begin{equation*}
-\int\limits_{0}^{t}f_{\varepsilon }\left( \left\Vert \nabla v_{l}\left(
\tau \right) \right\Vert _{L^{2}\left(
\mathbb{R}
^{n}\right) }\right) \int\limits_{B\left( 0,r\right) }\phi (x)\left\vert
\nabla v_{m}(\tau ,x)-\nabla v_{l}(\tau ,x)\right\vert ^{2}dxd\tau
\end{equation*}%
\begin{equation}
-\int\limits_{0}^{t}\tau \frac{d}{dt}\left( f_{\varepsilon }\left(
\left\Vert \nabla v_{l}\left( \tau \right) \right\Vert _{L^{2}\left(
\mathbb{R}
^{n}\right) }\right) \right) \int\limits_{B\left( 0,r\right) }\phi
(x)\left\vert \nabla v_{m}(\tau ,x)-\nabla v_{l}(\tau ,x)\right\vert
^{2}dxd\tau .  \tag{2.5}
\end{equation}%
By the conditions of the lemma and the definition of $f_{\varepsilon }$,\ it
follows that $\left\{ f_{\varepsilon }\left( \left\Vert \nabla v_{m}\left(
\cdot \right) \right\Vert _{L^{2}\left(
\mathbb{R}
^{n}\right) }\right) \right\} _{m=1}^{\infty }$ is bounded in $W^{1,\infty
}\left( 0,\infty \right) $. Then, considering (2.2) in (2.5), we get%
\begin{equation}
\underset{m\rightarrow \infty }{\text{ }\lim \text{ }}\underset{l\rightarrow
\infty }{\lim \sup }\left\vert \int\limits_{0}^{t}\tau f_{\varepsilon
}\left( \left\Vert \nabla v_{l}\left( \tau \right) \right\Vert _{L^{2}\left(
\mathbb{R}
^{n}\right) }\right) \frac{d}{d\tau }\int\limits_{B\left( 0,r\right) }\phi
(x)\left\vert \nabla v_{m}(\tau ,x)-\nabla v_{l}(\tau ,x)\right\vert
^{2}dxd\tau \right\vert =0.  \tag{2.6}
\end{equation}%
Taking into account (2.3), (2.4) and (2.6) in (2.1), we obtain%
\begin{equation*}
\underset{m\rightarrow \infty }{\lim \sup }\text{ }\underset{l\rightarrow
\infty }{\lim \sup }\left\vert \int\limits_{0}^{t}\int\limits_{B\left(
0,r\right) }\tau \left( f(\left\Vert \nabla v_{m}\left( t\right) \right\Vert
_{L^{2}\left(
\mathbb{R}
^{n}\right) }))\Delta v_{m}(t,x)-f(\left\Vert \nabla v_{l}\left( t\right)
\right\Vert _{L^{2}\left(
\mathbb{R}
^{n}\right) }))\Delta v_{l}(t,x)\right) \right.
\end{equation*}%
\begin{equation*}
\left. \times \phi (x)\left( v_{mt}\left( \tau ,x\right) -v_{lt}\left( \tau
,x\right) \right) dxd\tau \right\vert \leq c_{1}\text{ }t^{2}\max_{0\leq
s_{1},s_{2}\leq \varepsilon }\left\vert f\left( s_{1}\right) -f\left(
s_{2}\right) \right\vert \text{, \ }\forall t\geq 0\text{,}
\end{equation*}%
which yields the claim of the lemma, since $\varepsilon >0$ is arbitrary.
\end{proof}

\begin{lemma}
Assume that the condition (1.7) holds. Also, let the sequence $\left\{
v_{m}\right\} _{m=1}^{\infty }$ be weakly star convergent in $L^{\infty
}\left( 0,\infty ;H^{2}\left(
\mathbb{R}
^{n}\right) \right) $ and the sequence $\left\{ v_{mt}\right\}
_{m=1}^{\infty }$ be bounded in $L^{\infty }\left( 0,\infty ;L^{2}\left(
\mathbb{R}
^{n}\right) \right) $. Then, for every $r>0$ and $\phi \in L^{\infty
}(B\left( 0,r\right) )$
\begin{equation*}
\underset{m\rightarrow \infty }{\lim }\text{ }\underset{l\rightarrow \infty }%
{\lim }\int\limits_{0}^{t}\int\limits_{B\left( 0,r\right) }\tau (g\left(
v_{m}\left( \tau ,x\right) \right) -g\left( v_{l}\left( \tau ,x\right)
\right) )\phi (x)\left( v_{mt}\left( \tau ,x\right) -v_{lt}\left( \tau
,x\right) \right) dxd\tau =0\text{, \ \ }\forall t\geq 0\text{.}
\end{equation*}
\end{lemma}

\begin{proof}
We have%
\begin{equation*}
\int\limits_{0}^{t}\int\limits_{B\left( 0,r\right) }\tau \left( g\left(
v_{m}\left( \tau ,x\right) \right) -g\left( v_{l}\left( \tau ,x\right)
\right) \right) \phi (x)\left( v_{mt}\left( \tau ,x\right) -v_{lt}\left(
\tau ,x\right) \right) dxd\tau
\end{equation*}%
\begin{equation*}
=\int\limits_{0}^{t}\int\limits_{B\left( 0,r\right) }\tau \phi (x)g\left(
v_{m}\left( \tau ,x\right) \right) v_{mt}\left( \tau ,x\right) dxd\tau
+\int\limits_{0}^{t}\int\limits_{B\left( 0,r\right) }\tau \phi (x)g\left(
v_{l}\left( \tau ,x\right) \right) v_{lt}\left( \tau ,x\right) dxd\tau
\end{equation*}%
\begin{equation}
-\int\limits_{0}^{t}\int\limits_{B\left( 0,r\right) }\tau \phi (x)g\left(
v_{m}\left( \tau ,x\right) \right) v_{lt}\left( \tau ,x\right) dxd\tau
-\int\limits_{0}^{t}\int\limits_{B\left( 0,r\right) }\tau \phi (x)g\left(
v_{l}\left( \tau ,x\right) \right) v_{mt}\left( \tau ,x\right) dxd\tau .
\tag{2.7}
\end{equation}%
Let us estimate the first two terms on the right hand side of (2.7).
Applying integration by parts, we get%
\begin{equation*}
\int\limits_{0}^{t}\int\limits_{B\left( 0,r\right) }\tau \phi (x)g\left(
v_{m}\left( \tau ,x\right) \right) v_{mt}\left( \tau ,x\right) dxd\tau
+\int\limits_{0}^{t}\int\limits_{B\left( 0,r\right) }\tau \phi (x)g\left(
v_{l}\left( \tau ,x\right) \right) v_{lt}\left( \tau ,x\right) dxd\tau
\end{equation*}%
\begin{equation*}
=\int\limits_{0}^{t}\tau \frac{d}{d\tau }\left( \int\limits_{B\left(
0,r\right) }\phi (x)G\left( v_{m}\left( \tau ,x\right) \right) dx\right)
d\tau +\int\limits_{0}^{t}\tau \frac{d}{d\tau }\left( \int\limits_{B\left(
0,r\right) }\phi (x)G\left( v_{l}\left( \tau ,x\right) \right) dx\right)
d\tau
\end{equation*}%
\begin{equation*}
=t\int\limits_{B\left( 0,r\right) }\phi (x)G\left( v_{m}\left( t,x\right)
\right) dx+t\int\limits_{B\left( 0,r\right) }\phi (x)G\left( v_{l}\left(
\tau ,x\right) \right) dx
\end{equation*}%
\begin{equation}
-\int\limits_{0}^{t}\int\limits_{B\left( 0,r\right) }\phi (x)G\left(
v_{m}\left( \tau ,x\right) \right) dxd\tau
-\int\limits_{0}^{t}\int\limits_{B\left( 0,r\right) }\phi (x)G\left(
v_{l}\left( \tau ,x\right) \right) dxd\tau .  \tag{2.8}
\end{equation}%
By the conditions of the lemma, we obtain%
\begin{equation}
\left\{
\begin{array}{c}
v_{m}\rightarrow v\text{ weakly star in }L^{\infty }\left( 0,\infty
;H^{2}\left(
\mathbb{R}
^{n}\right) \right) \text{,} \\
v_{mt}\rightarrow v_{t}\text{ weakly star in }L^{\infty }\left( 0,\infty
;L^{2}\left(
\mathbb{R}
^{n}\right) \right) \text{,}%
\end{array}%
\right.  \tag{2.9}
\end{equation}%
for some $v\in L^{\infty }\left( 0,\infty ;H^{2}\left(
\mathbb{R}
^{n}\right) \right) \cap W^{1,\infty }\left( 0,\infty ;L^{2}\left(
\mathbb{R}
^{n}\right) \right) .$ Applying \cite[Corollary 4]{15}, by (2.9), we have%
\begin{equation*}
v_{m}\rightarrow v\text{ strongly in }C\left( \left[ 0,T\right]
;H^{2-\varepsilon }\left( B\left( 0,r\right) \right) \right) ,
\end{equation*}%
for every $\varepsilon >0$ and $T>0$. Hence, taking into account (1.7), we
get%
\begin{equation}
G\left( v_{m}\right) \rightarrow G\left( v\right) \text{ strongly in }%
C\left( [0,T];L^{1}\left( B\left( 0,r\right) \right) \right) \text{.}
\tag{2.10}
\end{equation}%
Then, passing to the limit in (2.8) and using (2.10), we obtain%
\begin{equation*}
\underset{m\rightarrow \infty }{\lim }\text{ }\underset{l\rightarrow \infty }%
{\lim }\left( \int\limits_{0}^{t}\int\limits_{B\left( 0,r\right) }\tau \phi
(x)g\left( v_{m}\left( \tau ,x\right) \right) v_{mt}\left( \tau ,x\right)
dxd\tau +\int\limits_{0}^{t}\int\limits_{B\left( 0,r\right) }\tau \phi
(x)g\left( v_{l}\left( \tau ,x\right) \right) v_{lt}\left( \tau ,x\right)
dxd\tau \right)
\end{equation*}%
\begin{equation}
=2t\int\limits_{B\left( 0,r\right) }\phi (x)G\left( v\left( t,x\right)
\right) dx-2\int\limits_{0}^{t}\int\limits_{B\left( 0,r\right) }\phi
(x)G\left( v\left( \tau ,x\right) \right) dxd\tau .  \tag{2.11}
\end{equation}%
Now, for the last two terms on the right hand side of (2.7), considering
(2.9), we get%
\begin{equation*}
\underset{m\rightarrow \infty }{\lim }\text{ }\underset{l\rightarrow \infty }%
{\lim }\left( -\int\limits_{0}^{t}\int\limits_{B\left( 0,r\right) }\tau \phi
(x)g\left( v_{m}\left( \tau ,x\right) \right) v_{lt}\left( \tau ,x\right)
dxd\tau -\int\limits_{0}^{t}\int\limits_{B\left( 0,r\right) }\tau \phi
(x)g\left( v_{l}\left( \tau ,x\right) \right) v_{mt}\left( \tau ,x\right)
dxd\tau \right)
\end{equation*}%
\begin{equation*}
=-2\int\limits_{0}^{t}\int\limits_{B\left( 0,r\right) }\tau \phi (x)g\left(
v\left( \tau ,x\right) \right) v_{t}\left( \tau ,x\right) dxd\tau
\end{equation*}%
\begin{equation}
=-2t\int\limits_{B\left( 0,r\right) }\phi (x)G\left( v\left( t,x\right)
\right) dx+2\int\limits_{0}^{t}\int\limits_{B\left( 0,r\right) }\phi
(x)G\left( v\left( \tau ,x\right) \right) dxd\tau .  \tag{2.12}
\end{equation}%
Hence, considering (2.11)-(2.12) and passing to the limit in (2.7), we
obtain the claim of the lemma.
\end{proof}

Now, we can prove the asymptotic compactness of $\left\{ S\left( t\right)
\right\} _{t\geq 0}$ in the interior domain.

\begin{theorem}
Assume that the conditions (1.3)-(1.8) hold and $\mathcal{B}$ is a bounded
subset of$\ H^{2}\left(
\mathbb{R}
^{n}\right) \times L^{2}\left(
\mathbb{R}
^{n}\right) $. Then every sequence of the form $\left\{ S(t_{k})\varphi
_{k}\right\} _{k=1}^{\infty },$ where $\left\{ \varphi _{k}\right\}
_{k=1}^{\infty }\subset \mathcal{B}$, $t_{k}\rightarrow \infty ,$ \ has a
convergent subsequence in $H^{2}\left( B\left( 0,r\right) \right) \times
L^{2}\left( B\left( 0,r\right) \right) $, for every $r>0$.
\end{theorem}

\begin{proof}
We will use the asymptotic compactness method introduced in \cite{16}.
Considering (1.3), (1.6), (1.7) and (1.8) in (1.9), we have
\begin{equation}
\underset{t\geq 0}{\sup }\underset{\varphi \in \mathcal{B}}{\sup }\left\Vert
S\left( t\right) \varphi \right\Vert _{H^{2}\left(
\mathbb{R}
^{n}\right) \times L^{2}\left(
\mathbb{R}
^{n}\right) }<\infty .  \tag{2.13}
\end{equation}%
Due to the boundedness of the sequence $\left\{ \varphi _{k}\right\}
_{k=1}^{\infty }$ in $H^{2}\left(
\mathbb{R}
^{n}\right) \times L^{2}\left(
\mathbb{R}
^{n}\right) $, by (2.13), it follows that the sequence $\left\{ S\left(
\cdot \right) \varphi _{k}\right\} _{k=1}^{\infty }$ is bounded in $%
L^{\infty }\left( 0,\infty ;H^{2}\left(
\mathbb{R}
^{n}\right) \times L^{2}\left(
\mathbb{R}
^{n}\right) \right) $. Then for any $T\geq 1$ \ there exists a subsequence $%
\left\{ k_{m}\right\} _{m=1}^{\infty }$ such that $t_{k_{m}}\geq T$, and
\begin{equation}
\left\{
\begin{array}{c}
v_{m}\rightarrow v\text{ weakly star in }L^{\infty }\left( 0,\infty
;H^{2}\left(
\mathbb{R}
^{n}\right) \right) \text{,} \\
v_{mt}\rightarrow v_{t}\text{ weakly star in }L^{\infty }\left( 0,\infty
;L^{2}\left(
\mathbb{R}
^{n}\right) \right) \text{,} \\
\left\Vert \nabla v_{m}\right\Vert _{L^{2}\left(
\mathbb{R}
^{n}\right) }^{2}\rightarrow q\text{ weakly star in }W^{1,\infty }\left(
0,\infty \right) \text{,} \\
v_{m}\rightarrow v\text{ strongly in }C\left( \left[ 0,T\right]
;H^{2-\varepsilon }\left( B\left( 0,r\right) \right) \right) \text{, \ }%
\varepsilon >0\text{,}%
\end{array}%
\right.  \tag{2.14}
\end{equation}%
for some $v\in $ $L^{\infty }\left( 0,\infty ;H^{2}\left(
\mathbb{R}
^{n}\right) \right) \cap W^{1,\infty }\left( 0,\infty ;L^{2}\left(
\mathbb{R}
^{n}\right) \right) $ and $q\in W^{1,\infty }\left( 0,\infty \right) $,
where $\left( v_{m}(t\right) ,v_{mt}\left( t\right)
)=S(t+t_{k_{m}}-T)\varphi _{k_{m}}$.

Now, taking into account (1.4) in (1.9), we find%
\begin{equation}
\int\limits_{0}^{\infty }\left\Vert v_{mt}(t)\right\Vert _{L^{2}\left(
\mathbb{R}
^{n}\backslash B\left( 0,r_{0}\right) \right)
}^{2}dt+\int\limits_{0}^{\infty }\left\Vert \nabla v_{mt}(t)\right\Vert
_{L^{2}\left(
\mathbb{R}
^{n}\backslash B\left( 0,r_{0}\right) \right) }^{2}dt\leq c_{1}\text{.\ }
\tag{2.15}
\end{equation}%
\textbf{\ }By (1.1), we have%
\begin{equation*}
v_{mtt}(t,x)-{div}\left( \beta \left( x\right) \nabla v_{mt}(t,x)\right)
+\gamma \Delta ^{2}v_{m}(t,x)+\alpha (x)v_{mt}(t,x)+\lambda v_{m}(t,x)
\end{equation*}%
\begin{equation}
=f(\left\Vert \nabla v_{m}\left( t\right) \right\Vert _{L^{2}\left(
\mathbb{R}
^{n}\right) }))\Delta v_{m}(t,x)-g\left( v_{m}(t,x)\right) +h\left( x\right)
.\newline
\tag{2.16}
\end{equation}%
Let $\eta \in C^{\infty }\left(
\mathbb{R}
^{n}\right) $, $0\leq \eta \left( x\right) \leq 1$, $\eta \left( x\right)
=\left\{
\begin{array}{c}
0,\text{ }\left\vert x\right\vert \leq 1\text{ } \\
1,\text{ }\left\vert x\right\vert \geq 2%
\end{array}%
\right. $ and $\eta _{r}\left( x\right) =\eta \left( \frac{x}{r}\right) $.
Multiplying (2.16) with $\eta _{r}^{2}v_{m}$ and integrating the obtained
equality over $\left( 0,T\right) \times
\mathbb{R}
^{n}$, we get%
\begin{equation*}
\int\limits_{0}^{T}\left( \gamma \left\Vert \eta _{r}\Delta
v_{m}(t)\right\Vert _{L^{2}\left(
\mathbb{R}
^{n}\right) }^{2}+\lambda \left\Vert \eta _{r}v_{m}(t)\right\Vert
_{L^{2}\left(
\mathbb{R}
^{n}\right) }^{2}\right) dt
\end{equation*}%
\begin{equation*}
=-\frac{1}{2}\int\limits_{%
\mathbb{R}
^{n}}\eta _{r}^{2}\left( x\right) \beta \left( x\right) \left\vert \nabla
v_{m}(T,x)\right\vert ^{2}dx+\frac{1}{2}\int\limits_{%
\mathbb{R}
^{n}}\eta _{r}^{2}\left( x\right) \beta \left( x\right) \left\vert \nabla
v_{m}(0,x)\right\vert ^{2}dx
\end{equation*}%
\begin{equation*}
-\frac{2}{r}\sum_{i=1}^{n}\int\limits_{0}^{T}\int\limits_{%
\mathbb{R}
^{n}}\beta \left( x\right) v_{mtx_{i}}(t,x)\eta _{r}\eta _{x_{i}}\left(
\frac{x}{r}\right) v_{m}(t,x)dxdt
\end{equation*}%
\begin{equation*}
+\int\limits_{0}^{T}\left\Vert \eta _{r}v_{mt}\left( t\right) \right\Vert
_{L^{2}\left(
\mathbb{R}
^{n}\right) }^{2}dt-\int\limits_{%
\mathbb{R}
^{n}}\eta _{r}^{2}\left( x\right) v_{mt}\left( T,x\right)
v_{m}(T,x)dx+\int\limits_{%
\mathbb{R}
^{n}}\eta _{r}^{2}\left( x\right) v_{mt}\left( 0,x\right) v_{m}(0,x)dx
\end{equation*}%
\begin{equation*}
-\frac{4\gamma }{r}\sum_{i=1}^{n}\int\limits_{0}^{T}\int\limits_{%
\mathbb{R}
^{n}}\eta _{r}\left( x\right) \eta _{x_{i}}\left( \frac{x}{r}\right) \Delta
v_{m}(t,x)v_{mx_{i}}(t,x)dxdt-\gamma \int\limits_{0}^{T}\int\limits_{%
\mathbb{R}
^{n}}\Delta \left( \eta _{r}^{2}\left( x\right) \right) \Delta
v_{m}(t,x)v_{m}(t,x)dxdt
\end{equation*}%
\begin{equation*}
-\frac{1}{2}\int\limits_{%
\mathbb{R}
^{n}}\eta _{r}^{2}\left( x\right) \alpha \left( x\right) \left\vert
v_{m}(T,x)\right\vert ^{2}dx+\frac{1}{2}\int\limits_{%
\mathbb{R}
^{n}}\eta _{r}^{2}\left( x\right) \alpha \left( x\right) \left\vert
v_{m}(0,x)\right\vert ^{2}dx
\end{equation*}%
\begin{equation*}
-\int\limits_{0}^{T}f(\left\Vert \nabla v_{m}\left( t\right) \right\Vert
_{L^{2}\left(
\mathbb{R}
^{n}\right) }))\int\limits_{%
\mathbb{R}
^{n}}\eta _{r}^{2}\left( x\right) \left\vert \nabla v_{m}(t,x)\right\vert
^{2}dxdt
\end{equation*}%
\begin{equation*}
-\frac{2}{r}\sum_{i=1}^{n}\int\limits_{0}^{T}f(\left\Vert \nabla v_{m}\left(
t\right) \right\Vert _{L^{2}\left(
\mathbb{R}
^{n}\right) }))\int\limits_{%
\mathbb{R}
^{n}}\eta _{r}\eta _{x_{i}}\left( \frac{x}{r}\right) v_{mx_{i}}(t,x)v_{m}dxdt
\end{equation*}%
\begin{equation*}
-\int\limits_{0}^{T}\int\limits_{%
\mathbb{R}
^{n}}g\left( v_{m}\left( t,x\right) \right) \eta _{r}^{2}\left( x\right)
v_{m}\left( t,x\right) dxdt
\end{equation*}%
\begin{equation}
+\int\limits_{0}^{T}\int\limits_{%
\mathbb{R}
^{n}}h\left( x\right) \eta _{r}^{2}\left( x\right) v_{m}\left( t,x\right)
dxdt.  \tag{2.17}
\end{equation}%
Taking into account (1.3), (1.6), (1.8), (1.9), (2.13) and (2.15) in (2.17),
we obtain%
\begin{equation*}
\lim \sup_{m\rightarrow \infty }\int\limits_{0}^{T}\left( \gamma \left\Vert
\Delta v_{m}(t)\right\Vert _{L^{2}\left(
\mathbb{R}
^{n}\backslash B\left( 0,2r\right) \right) }^{2}+\lambda \left\Vert
v_{m}(t)\right\Vert _{L^{2}\left(
\mathbb{R}
^{n}\backslash B\left( 0,2r\right) \right) }^{2}\right) dt
\end{equation*}%
\begin{equation}
\leq c_{2}\left( 1+\frac{\sqrt{T}}{r}+\frac{T}{r}+T\left\Vert h\right\Vert
_{L^{2}\left(
\mathbb{R}
^{n}\backslash B\left( 0,r\right) \right) }\right) \text{, \ }\forall r\geq
r_{0}\text{.}  \tag{2.18}
\end{equation}%
Now, by (1.1), we have%
\begin{equation*}
v_{mtt}(t,x)-v_{ltt}(t,x)-{div}\left( \beta \left( x\right) \cdot \nabla
\left( v_{mt}(t,x)-v_{lt}(t,x)\right) \right) +\gamma \Delta ^{2}\left(
v_{m}(t,x)-v_{l}(t,x)\right)
\end{equation*}%
\begin{equation*}
+\alpha (x)\left( v_{mt}(t,x)-v_{lt}(t,x)\right) +\lambda \left(
v_{m}(t,x)-v_{l}(t,x)\right)
\end{equation*}%
\begin{equation}
=f(\left\Vert \nabla v_{m}\left( t\right) \right\Vert _{L^{2}\left(
\mathbb{R}
^{n}\right) }))\Delta v_{m}(t,x)-f(\left\Vert \nabla v_{l}\left( t\right)
\right\Vert _{L^{2}\left(
\mathbb{R}
^{n}\right) }))\Delta v_{l}(t,x)-g\left( v_{m}\right) +g\left( v_{l}\right) .%
\newline
\tag{2.19}
\end{equation}%
\ Multiplying (2.19) by $\sum\nolimits_{i=1}^{n}x_{i}\left( 1-\eta
_{4r}\right) \left( v_{m}-v_{l}\right) _{x_{i}}+\frac{1}{2}\left( n-1\right)
\left( 1-\eta _{4r}\right) \left( v_{m}-v_{l}\right) $, integrating the
obtained equality over $\left( 0,T\right) \times
\mathbb{R}
^{n}$ and taking into account (2.13), we obtain%
\begin{equation*}
\frac{3\gamma }{2}\int\limits_{0}^{T}\left\Vert \Delta \left( v_{m}\left(
t\right) -v_{l}\left( t\right) \right) \right\Vert _{L^{2}\left( B\left(
0,4r\right) \right) }^{2}dt+\frac{1}{2}\int\limits_{0}^{T}\left\Vert
v_{mt}\left( t\right) -v_{lt}\left( t\right) \right\Vert _{L^{2}\left(
B\left( 0,4r\right) \right) }^{2}dt
\end{equation*}%
\begin{equation*}
\leq c_{3}(1+T+rT)\left\Vert v_{m}-v_{l}\right\Vert _{C[0,T];H^{1}(B\left(
0,8r\right) )}
\end{equation*}%
\begin{equation*}
+c_{3}\left( \sqrt{T}+r\sqrt{T}\right) \left\Vert \sqrt{\beta }\left( \nabla
v_{mt}-\nabla v_{lt}\right) \right\Vert _{L^{2}(\left( 0,T\right) \times
B\left( 0,8r\right) )}
\end{equation*}%
\begin{equation}
+c_{3}\left( \left\Vert v_{mt}-v_{lt}\right\Vert _{L^{2}\left(
0,T;L^{2}\left( B\left( 0,8r\right) \backslash B\left( 0,4r\right) \right)
\right) }^{2}+\left\Vert v_{m}-v_{l}\right\Vert _{L^{2}\left(
0,T;H^{2}\left( B\left( 0,8r\right) \backslash B\left( 0,4r\right) \right)
\right) }^{2}\right) .  \tag{2.20}
\end{equation}%
Thus, considering (2.14), (2.15), (2.18) and passing to the limit in (2.20)
, we get
\begin{equation*}
\underset{m\rightarrow \infty }{\lim \sup }\text{ }\underset{l\rightarrow
\infty }{\lim \sup }\int\limits_{0}^{T}\left[ \left\Vert \Delta \left(
v_{m}\left( t\right) -v_{l}\left( t\right) \right) \right\Vert _{L^{2}\left(
B\left( 0,4r\right) \right) }^{2}+\left\Vert v_{mt}\left( t\right)
-v_{lt}\left( t\right) \right\Vert _{L^{2}\left( B\left( 0,4r\right) \right)
}^{2}\right] dt
\end{equation*}%
\begin{equation}
\leq c_{4}\left( 1+\frac{\sqrt{T}}{r}+\frac{T}{r}+r\sqrt{T}+T\left\Vert
h\right\Vert _{L^{2}\left(
\mathbb{R}
^{n}\backslash B\left( 0,2r\right) \right) }\right) \text{, \ }\forall r\geq
r_{0}\text{.}  \tag{2.21}
\end{equation}%
Now, multiplying (2.19) by $\left( 1-\eta _{2r}\right) ^{4}t\left[ 2\left(
v_{mt}-v_{lt}\right) +\alpha _{0}\eta _{r}^{4}\left( v_{m}-v_{l}\right) %
\right] $ and integrating the obtained equality over $\left( 0,T\right)
\times
\mathbb{R}
^{n},$ we obtain%
\begin{equation*}
\gamma T\left\Vert \Delta \left( v_{m}\left( T\right) -v_{l}\left( T\right)
\right) \right\Vert _{L^{2}\left( B\left( 0,2r\right) \right)
}^{2}+T\left\Vert v_{mt}\left( T\right) -v_{lt}\left( T\right) \right\Vert
_{L^{2}\left( B\left( 0,2r\right) \right) }^{2}+
\end{equation*}%
\begin{equation*}
+T\lambda \left\Vert v_{m}\left( T\right) -v_{l}\left( T\right) \right\Vert
_{L^{2}\left( B\left( 0,2r\right) \right) }^{2}\leq
\int\limits_{0}^{T}\left\Vert v_{mt}\left( t\right) -v_{lt}\left( t\right)
\right\Vert _{L^{2}\left( B\left( 0,4r\right) \right) }^{2}dt
\end{equation*}%
\begin{equation*}
+\gamma \int\limits_{0}^{T}\left\Vert \Delta \left( v_{m}(t)-v_{l}(t)\right)
\right\Vert _{L^{2}\left( B\left( 0,4r\right) \right) }^{2}dt+\lambda
\int\limits_{0}^{T}\left\Vert v_{m}(t)-v_{l}(t)\right\Vert _{L^{2}\left(
B\left( 0,4r\right) \right) }^{2}dt
\end{equation*}%
\begin{equation*}
+2\left\vert \int\limits_{0}^{T}\int\limits_{B\left( 0,4r\right) }t\left(
f(\left\Vert \nabla v_{m}\left( t\right) \right\Vert _{L^{2}\left(
\mathbb{R}
^{n}\right) })\Delta v_{m}\left( t,x\right) -f(\left\Vert \nabla v_{l}\left(
t\right) \right\Vert _{L^{2}\left(
\mathbb{R}
^{n}\right) })\Delta v_{l}\left( t,x\right) )\right) \right.
\end{equation*}%
\begin{equation*}
\left. \times \left( 1-\eta _{2r}\right) ^{4}\left( v_{mt}\left( t,x\right)
-v_{lt}\left( t,x\right) \right) dxdt\right\vert
\end{equation*}%
\begin{equation*}
+2\left\vert \int\limits_{0}^{T}\int\limits_{B\left( 0,4r\right) }t\left(
g\left( v_{m}\left( t,x\right) \right) -g\left( v_{l}\left( t,x\right)
\right) \right) \left( 1-\eta _{2r}\right) ^{4}\left( v_{mt}\left(
t,x\right) -v_{lt}\left( t,x\right) \right) dxdt\right\vert
\end{equation*}%
\begin{equation*}
+\alpha _{0}\left\vert \int\limits_{0}^{T}\int\limits_{B\left( 0,4r\right)
}t\left( f(\left\Vert \nabla v_{m}\left( t\right) \right\Vert _{L^{2}\left(
\mathbb{R}
^{n}\right) })\Delta v_{m}\left( t,x\right) -f(\left\Vert \nabla v_{l}\left(
t\right) \right\Vert _{L^{2}\left(
\mathbb{R}
^{n}\right) })\Delta v_{l}\left( t,x\right) )\right) \right.
\end{equation*}%
\begin{equation*}
\left. \times \left( 1-\eta _{2r}\right) ^{4}\eta _{r}^{4}(x)\left(
v_{m}\left( t,x\right) -v_{l}\left( t,x\right) \right) dxdt\right\vert
\end{equation*}%
\begin{equation*}
+\alpha _{0}\left\vert \int\limits_{0}^{T}\int\limits_{B\left( 0,4r\right)
}t\left( g\left( v_{m}\left( t,x\right) \right) -g\left( v_{l}\left(
t,x\right) \right) \right) \left( 1-\eta _{2r}\right) ^{4}\eta
_{r}^{4}(x)\left( v_{m}\left( t,x\right) -v_{l}\left( t,x\right) \right)
dxdt\right\vert
\end{equation*}%
\begin{equation*}
+\frac{c_{5}T}{r}\int\limits_{0}^{T}\int\limits_{B\left( 0,4r\right)
\backslash B\left( 0,r\right) }\left\vert \nabla \left( v_{mt}\left(
t,x\right) -v_{lt}\left( t,x\right) \right) \right\vert ^{2}dxdt
\end{equation*}%
\begin{equation*}
+\frac{c_{5}T}{r}\int\limits_{0}^{T}\int\limits_{B\left( 0,4r\right)
\backslash B\left( 0,r\right) }\left\vert v_{mt}\left( t,x\right)
-v_{lt}\left( t,x\right) \right\vert ^{2}dxdt
\end{equation*}%
\begin{equation}
+c_{5}T\left\Vert v_{m}-v_{l}\right\Vert _{C([0,T];H^{1}\left( B\left(
0,4r\right) \right) )}^{2}\text{, \ \ }\forall r\geq r_{0}\text{, \ }\forall
T\geq 1\text{.}  \tag{2.22}
\end{equation}%
Then, taking into account (2.14), (2.15), (2.21), Lemma 2.1 and Lemma 2.2,
and passing to the limit in (2.22), we find%
\begin{equation*}
\underset{m\rightarrow \infty }{\lim \sup }\text{ }\underset{l\rightarrow
\infty }{\lim \sup }\left( \left\Vert v_{m}\left( T\right) -v_{l}\left(
T\right) \right\Vert _{H^{2}\left( B\left( 0,2r\right) \right)
}^{2}+\left\Vert v_{mt}\left( T\right) -v_{lt}\left( T\right) \right\Vert
_{L^{2}\left( B\left( 0,2r\right) \right) }^{2}\right)
\end{equation*}%
\begin{equation}
\leq c_{6}\left( \frac{1}{T}+\frac{1}{\sqrt{T}r}+\frac{1}{r}+\frac{r}{\sqrt{T%
}}+\left\Vert h\right\Vert _{L^{2}\left(
\mathbb{R}
^{n}\backslash B\left( 0,2r\right) \right) }\right) \text{, \ }\forall r\geq
r_{0}\text{, \ }\forall T\geq 1\text{.}  \tag{2.23}
\end{equation}%
Thus, by the definition of $v_{m}$, the inequality (2.23) yields%
\begin{equation*}
\underset{m\rightarrow \infty }{\lim \sup }\text{ }\underset{l\rightarrow
\infty }{\lim \sup }\left\Vert S(t_{k_{m}})\varphi
_{k_{m}}-S(t_{k_{l}})\varphi _{k_{l}}\right\Vert _{H^{2}\left( B\left(
0,r\right) \right) \times L^{2}\left( B\left( 0,r\right) \right) }^{2}
\end{equation*}%
\begin{equation}
\leq c_{7}\left( \frac{1}{T}+\frac{1}{\sqrt{T}r}+\frac{1}{r}+\frac{r}{\sqrt{T%
}}+\left\Vert h\right\Vert _{L^{2}\left(
\mathbb{R}
^{n}\backslash B\left( 0,r\right) \right) }\right) \text{, \ }\forall r\geq
2r_{0}\text{, \ }\forall T\geq 1\text{.}  \tag{2.24}
\end{equation}%
Passing to the limit as $T\rightarrow \infty $ in (2.24), we obtain%
\begin{equation*}
\underset{l\rightarrow \infty }{\lim \inf }\text{ }\underset{m\rightarrow
\infty }{\lim \inf }\left\Vert S(t_{k})\varphi _{k}-S(t_{m})\varphi
_{m}\right\Vert _{H^{2}\left( B\left( 0,r\right) \right) \times L^{2}\left(
B\left( 0,r\right) \right) }^{2}
\end{equation*}%
\begin{equation*}
\leq c_{7}\left( \frac{1}{r}+\left\Vert h\right\Vert _{L^{2}\left(
\mathbb{R}
^{n}\backslash B\left( 0,r\right) \right) }\right) \text{, \ }\forall r\geq
2r_{0},
\end{equation*}%
which gives%
\begin{equation*}
\underset{l\rightarrow \infty }{\lim \inf }\text{ }\underset{m\rightarrow
\infty }{\lim \inf }\left\Vert S(t_{k})\varphi _{k}-S(t_{m})\varphi
_{m}\right\Vert _{H^{2}\left( B\left( 0,r\right) \right) \times L^{2}\left(
B\left( 0,r\right) \right) }^{2}
\end{equation*}%
\begin{equation}
\leq c_{7}\left( \frac{1}{\widetilde{r}}+\left\Vert h\right\Vert
_{L^{2}\left(
\mathbb{R}
^{n}\backslash B\left( 0,\widetilde{r}\right) \right) }\right) \text{, \ \ }%
\forall \widetilde{r}\geq r\geq 2r_{0}.  \tag{2.25}
\end{equation}%
Consequently, by passing to the limit as $\widetilde{r}\rightarrow \infty $
in (2.25), we deduce%
\begin{equation}
\underset{l\rightarrow \infty }{\lim \inf }\text{ }\underset{m\rightarrow
\infty }{\lim \inf }\left\Vert S(t_{k})\varphi _{k}-S(t_{m})\varphi
_{m}\right\Vert _{H^{2}\left( B\left( 0,r\right) \right) \times L^{2}\left(
B\left( 0,r\right) \right) }=0\text{, \ \ }\forall r>0\text{.}  \tag{2.26}
\end{equation}%
Let $r_{i}\nearrow \infty $ as $i\rightarrow \infty $. Taking $r=r_{i}$ in
(2.26) and using the arguments at the end of the proof of \cite[Lemma 3.4]%
{17}, we can say that there exist subsequences $\left\{ k_{m}^{(i)}\right\} $
\ such that
\begin{equation*}
\left\{ k_{m}^{(1)}\right\} \supset \left\{ k_{m}^{(2)}\right\} \supset
...\supset \left\{ k_{m}^{(i)}\right\} \supset ..
\end{equation*}%
and
\begin{equation*}
\left\{ S(t_{k_{m}^{(i)}})\varphi _{k_{m}^{(i)}}\right\} \text{ converges in
}H^{2}\left( B\left( 0,r_{i}\right) \right) \times L^{2}\left( B\left(
0,r_{i}\right) \right) \text{.}
\end{equation*}%
Thus, the diagonal subsequence $\left\{ S(t_{k_{m}^{(m)}})\varphi
_{k_{m}^{(m)}}\right\} $ converges in $H^{2}\left( B\left( 0,r\right)
\right) \times L^{2}\left( B\left( 0,r\right) \right) $, for every $r>0$.
\end{proof}

To establish the tail estimate, we need the following lemma.

\begin{lemma}
Let the conditions (1.3)-(1.6) hold and $B$ be a bounded subset of $%
H^{2}\left(
\mathbb{R}
^{n}\right) .$ Then for every $\varepsilon >0$ there exist a constant $%
\delta \equiv \delta \left( \varepsilon \right) >0$ and functions $\psi
_{\varepsilon }\in L^{\infty }\left(
\mathbb{R}
^{n}\right) $, $\varphi _{\varepsilon }\in C^{\infty }\left(
\mathbb{R}
^{n}\right) $, such that $0\leq \psi _{\varepsilon }\leq \min \left\{
1,\delta ^{-1}\beta \right\} $ a.e. in $%
\mathbb{R}
^{n}$, $0\leq \varphi _{\varepsilon }\leq 1$ in $%
\mathbb{R}
^{n}$, supp$\left( \varphi _{\varepsilon }\right) \subset \{x\in
\mathbb{R}
^{n}:\alpha \left( x\right) \geq \delta $ a.e. in $%
\mathbb{R}
^{n}\}$ and%
\begin{equation}
\left\vert f\left( \left\Vert \nabla u\right\Vert _{L^{2}\left(
\mathbb{R}
^{n}\right) }\right) -f_{\delta }\left( \sqrt{\left\Vert \sqrt{\psi
_{\varepsilon }}\nabla u\right\Vert _{L^{2}\left(
\mathbb{R}
^{n}\right) }^{2}+\left\Vert \sqrt{\varphi _{\varepsilon }}\nabla
u\right\Vert _{L^{2}\left(
\mathbb{R}
^{n}\right) }^{2}}\right) \right\vert <\varepsilon ,  \tag{2.27}
\end{equation}%
for every $u\in B,$ where $f_{\delta }$ is the function defined in the proof
of Lemma 2.1.
\end{lemma}

\begin{proof}
Let $A_{0}=\left\{ x\in B\left( 0,r_{0}\right) :\alpha \left( x\right)
=0\right\} $ and $A_{k}=\left\{ x\in B\left( 0,r_{0}\right) :0\leq \alpha
\left( x\right) <\frac{1}{k}\right\} $. It is easy to see that $%
A_{k+1}\subset A_{k}$, and $A_{0}=\underset{k>0}{\cap }A_{k}.$ Hence, $%
\underset{k\rightarrow \infty }{\lim }mes\left( A_{k}\right) =mes\left(
A_{0}\right) $. So, for $\delta >0$, there exists $k_{\delta }~$such that%
\begin{equation}
mes\left( A_{k_{\delta }}\backslash A_{0}\right) <\frac{\delta }{3}.
\tag{2.28}
\end{equation}%
Since $A_{k_{\delta }}$ is a measurable subset of $B\left( 0,r_{0}\right) ,$
there exists an open set $O_{\delta }^{\left( 1\right) }\subset B\left(
0,r_{0}\right) $ such that $A_{k_{\delta }}\subset O_{\delta }^{\left(
1\right) }$ and
\begin{equation}
mes\left( O_{\delta }^{\left( 1\right) }\backslash A_{k_{\delta }}\right) <%
\frac{\delta }{3}.  \tag{2.29}
\end{equation}%
Now, let $\eta _{\delta }\in C_{0}\left(
\mathbb{R}
^{n}\right) $ such that $0\leq \eta _{\delta }\leq 1,$ $\left. \eta _{\delta
}\right\vert _{O_{\delta }^{\left( 1\right) }}=1$ and supp$\left( \eta
_{\delta }\right) \subset $ $O_{\delta }^{\left( 2\right) },$where $%
O_{\delta }^{\left( 1\right) }\Subset $ $O_{\delta }^{\left( 2\right) }$ and
\begin{equation}
mes\left( O_{\delta }^{\left( 2\right) }\backslash O_{\delta }^{\left(
1\right) }\right) <\frac{\delta }{3}.  \tag{2.30}
\end{equation}%
Then setting $\varphi _{\delta }:=1-$ $\eta _{\delta }$, we have $\varphi
_{\delta }\in C\left(
\mathbb{R}
^{n}\right) $, $0\leq \varphi _{\delta }\leq 1,\left. \varphi _{\delta
}\right\vert _{%
\mathbb{R}
^{n}\backslash O_{\delta }^{\left( 2\right) }}=1$ and supp$(\varphi _{\delta
})\subset
\mathbb{R}
^{n}\backslash O_{\delta }^{\left( 1\right) }$.

By (2.28)-(2.30), we obtain%
\begin{equation*}
\left\vert \int\limits_{O_{\delta }^{\left( 2\right) }}\varphi _{\delta
}\left\vert \nabla u\left( x\right) \right\vert
^{2}dx-\int\limits_{O_{\delta }^{\left( 2\right) }\backslash
A_{0}}\left\vert \nabla u\left( x\right) \right\vert ^{2}dx\right\vert
\end{equation*}%
\begin{equation*}
=\left\vert \int\limits_{O_{\delta }^{\left( 2\right) }\backslash O_{\delta
}^{\left( 1\right) }}\varphi _{\delta }\left\vert \nabla u\left( x\right)
\right\vert ^{2}dx-\int\limits_{O_{\delta }^{\left( 2\right) }\backslash
A_{0}}\left\vert \nabla u\left( x\right) \right\vert ^{2}dx\right\vert
\end{equation*}%
\begin{equation*}
\leq 2\int\limits_{O_{\delta }^{\left( 2\right) }\backslash A_{0}}\left\vert
\nabla u\left( x\right) \right\vert ^{2}dx\leq 2c\left\Vert u\right\Vert
_{H^{2}\left(
\mathbb{R}
^{n}\right) }^{2}\left( mes\left( O_{\delta }^{\left( 2\right) }\backslash
A_{0}\right) \right) ^{n^{\ast }}
\end{equation*}%
\begin{equation}
<2c\delta ^{n^{\ast }}\left\Vert u\right\Vert _{H^{2}\left(
\mathbb{R}
^{n}\right) }^{2},  \tag{2.31}
\end{equation}%
for every $u\in H^{2}\left(
\mathbb{R}
^{n}\right) $, where $n^{\ast }=\left\{
\begin{array}{c}
1,\text{ \ \ \ \ \ \ \ \ \ \ \ \ \ \ \ \ }n=1, \\
q,\text{ \ }0<q<1,\text{ \ }n=2, \\
\frac{2}{n},\text{ \ \ \ \ \ \ \ \ \ \ \ \ \ \ \ \ }n\geq 3%
\end{array}%
\right. $ and $c>0$.

Now, by (1.5), it follows that%
\begin{equation*}
\beta >0\text{ \ a.e. in }A_{0}.
\end{equation*}%
Hence, by Lebesgue dominated convergence theorem, there exists $\lambda
_{\delta }>0$ such that%
\begin{equation*}
\int\limits_{A_{0}}\frac{\lambda _{\delta }}{\lambda _{\delta }+\beta \left(
x\right) }dx<\delta ,
\end{equation*}%
which yields
\begin{equation}
\left\vert \int\limits_{A_{0}}\left\vert \nabla u\left( x\right) \right\vert
^{2}dx-\int\limits_{A_{0}}\frac{\beta \left( x\right) }{\lambda _{\delta
}+\beta \left( x\right) }\left\vert \nabla u\left( x\right) \right\vert
^{2}dx\right\vert <c\delta ^{n^{\ast }}\left\Vert u\right\Vert _{H^{2}\left(
\mathbb{R}
^{n}\right) }^{2}.  \tag{2.32}
\end{equation}%
Thus, denoting $\psi _{\delta }=\left\{
\begin{array}{c}
\frac{\beta \left( x\right) }{\lambda _{\delta }+\beta \left( x\right) }%
\text{, }x\in A_{0}, \\
0\text{, }x\in
\mathbb{R}
^{n}\backslash A_{0},%
\end{array}%
\right. $ by (2.31) and (2.32), we get%
\begin{equation*}
\left\vert \left\Vert \nabla u\right\Vert _{L^{2}\left(
\mathbb{R}
^{n}\right) }^{2}-\left\Vert \sqrt{\psi _{\delta }}\nabla u\right\Vert
_{L^{2}\left(
\mathbb{R}
^{n}\right) }^{2}-\left\Vert \sqrt{\varphi _{\delta }}\nabla u\right\Vert
_{L^{2}\left(
\mathbb{R}
^{n}\right) }^{2}\right\vert
\end{equation*}%
\begin{equation*}
<3c\delta ^{n^{\ast }}\left\Vert u\right\Vert _{H^{2}\left(
\mathbb{R}
^{n}\right) }^{2},
\end{equation*}%
and consequently%
\begin{equation*}
\left\vert \left\Vert \nabla u\right\Vert _{L^{2}\left(
\mathbb{R}
^{n}\right) }-\sqrt{\left\Vert \sqrt{\psi _{\delta }}\nabla u\right\Vert
_{L^{2}\left(
\mathbb{R}
^{n}\right) }^{2}+\left\Vert \sqrt{\varphi _{\delta }}\nabla u\right\Vert
_{L^{2}\left(
\mathbb{R}
^{n}\right) }^{2}}\right\vert
\end{equation*}%
\begin{equation*}
\leq \sqrt{\left\vert \left\Vert \nabla u\right\Vert _{L^{2}\left(
\mathbb{R}
^{n}\right) }^{2}-\left\Vert \sqrt{\psi _{\delta }}\nabla u\right\Vert
_{L^{2}\left(
\mathbb{R}
^{n}\right) }^{2}-\left\Vert \sqrt{\varphi _{\delta }}\nabla u\right\Vert
_{L^{2}\left(
\mathbb{R}
^{n}\right) }^{2}\right\vert }
\end{equation*}%
\begin{equation*}
<\sqrt{3c}\delta ^{\frac{1}{2}n^{\ast }}\left\Vert u\right\Vert
_{H^{2}\left(
\mathbb{R}
^{n}\right) }.
\end{equation*}%
The last inequality, together with the differentiability of the function $f$%
, yields (2.27).
\end{proof}

Now, let us proof the following tail estimate.

\begin{theorem}
\bigskip Assume that the conditions (1.3)-(1.8) hold and $\mathcal{B}$ is a
bounded subset of$\ H^{2}\left(
\mathbb{R}
^{n}\right) \times L^{2}\left(
\mathbb{R}
^{n}\right) $. Then for any $\varepsilon >0$ there exist $T\equiv T\left(
\mathcal{B},\varepsilon \right) $ and $R\equiv R\left( \mathcal{B}%
,\varepsilon \right) $ such that%
\begin{equation*}
\left\Vert S\left( t\right) \varphi \right\Vert _{H^{2}\left(
\mathbb{R}
^{n}\backslash B\left( 0,r\right) \right) \times L^{2}\left(
\mathbb{R}
^{n}\backslash B\left( 0,r\right) \right) }<\varepsilon \,\text{,}
\end{equation*}%
for every $t\geq T$, $r\geq R$ and $\varphi \in \mathcal{B}.$
\end{theorem}

\begin{proof}
Let $\left( u_{0},u_{1}\right) \in \mathcal{B}$ and $\left( u\left( t\right)
,u_{t}\left( t\right) \right) =S\left( t\right) \left( u_{0},u_{1}\right) $.
Multiplying (1.1) with $\eta _{r}^{2}u_{t}$, integrating the obtained
equality over $%
\mathbb{R}
^{n}$ and taking into account (2.13)$,$ we get%
\begin{equation*}
\frac{1}{2}\frac{d}{dt}\left( \left\Vert \eta _{r}u_{t}\left( t\right)
\right\Vert _{L^{2}\left(
\mathbb{R}
^{n}\right) }^{2}+\gamma \left\Vert \eta _{r}\Delta u\left( t\right)
\right\Vert _{L^{2}\left(
\mathbb{R}
^{n}\right) }^{2}+\lambda \left\Vert \eta _{r}u\left( t\right) \right\Vert
_{L^{2}\left(
\mathbb{R}
^{n}\right) }^{2}\right)
\end{equation*}%
\begin{equation*}
+\frac{d}{dt}\left( \int\limits_{%
\mathbb{R}
^{n}}\eta _{r}^{2}\left( x\right) G\left( u\left( t,x\right) \right)
dx\right) +\left\Vert \sqrt{\beta }\eta _{r}\nabla u_{t}\left( t\right)
\right\Vert _{L^{2}\left(
\mathbb{R}
^{n}\right) }^{2}+\left\Vert \sqrt{\alpha }\eta _{r}u_{t}\left( t\right)
\right\Vert _{L^{2}\left(
\mathbb{R}
^{n}\right) }^{2}
\end{equation*}%
\begin{equation*}
-\int\limits_{%
\mathbb{R}
^{n}}f(\left\Vert \nabla u\left( t\right) \right\Vert _{L^{2}\left(
\mathbb{R}
^{n}\right) })\Delta u\eta _{r}^{2}u_{t}dx
\end{equation*}%
\begin{equation}
\leq c_{2}\left( \frac{1}{r}+\frac{1}{r}\left\Vert \sqrt{\beta }\eta
_{r}\nabla u_{t}\left( t\right) \right\Vert _{L^{2}\left( B\left(
0,2r\right) \right) }+\left\Vert h\right\Vert _{L^{2}\left(
\mathbb{R}
^{n}\backslash B\left( 0,r\right) \right) }\right) \text{, \ \ }\forall
r\geq r_{0}\text{.}  \tag{2.33}
\end{equation}%
Now, let us estimate the last term on the left hand side of (2.33). By Lemma
2.3, we have%
\begin{equation*}
-f(\left\Vert \nabla u\left( t\right) \right\Vert _{L^{2}\left(
\mathbb{R}
^{n}\right) })\int\limits_{%
\mathbb{R}
^{n}}\Delta u\eta _{r}^{2}u_{t}dx
\end{equation*}%
\begin{equation*}
\geq -\varepsilon \left\Vert \eta _{r}\Delta u\left( t\right) \right\Vert
_{L^{2}\left(
\mathbb{R}
^{n}\right) }\left\Vert \eta _{r}u_{t}\left( t\right) \right\Vert
_{L^{2}\left(
\mathbb{R}
^{n}\right) }-\frac{c_{3}}{r}
\end{equation*}%
\begin{equation}
+\frac{1}{2}f_{\delta }\left( \sqrt{\left\Vert \sqrt{\psi _{\varepsilon }}%
\nabla u\left( t\right) \right\Vert _{L^{2}\left(
\mathbb{R}
^{n}\right) }^{2}+\left\Vert \sqrt{\phi _{\varepsilon }}\nabla u\left(
t\right) \right\Vert _{L^{2}\left(
\mathbb{R}
^{n}\right) }^{2}}\right) \frac{d}{dt}\left( \left\Vert \eta _{r}\nabla
\left( u\left( t\right) \right) \right\Vert _{L^{2}\left(
\mathbb{R}
^{n}\right) }^{2}\right) .  \tag{2.34}
\end{equation}%
Moreover, for the last term on the right hand side of (2.34), by using the
definition of $f_{\delta }$ and the properties of $\psi _{\varepsilon }$ and
$\varphi _{\varepsilon }$, we obtain%
\begin{equation*}
\frac{1}{2}f_{\delta }\left( \sqrt{\left\Vert \sqrt{\psi _{\varepsilon }}%
\nabla u\left( t\right) \right\Vert _{L^{2}\left(
\mathbb{R}
^{n}\right) }^{2}+\left\Vert \sqrt{\varphi _{\varepsilon }}\nabla u\left(
t\right) \right\Vert _{L^{2}\left(
\mathbb{R}
^{n}\right) }^{2}}\right) \frac{d}{dt}\left( \left\Vert \eta _{r}\nabla
\left( u\left( t\right) \right) \right\Vert _{L^{2}\left(
\mathbb{R}
^{n}\right) }^{2}\right)
\end{equation*}%
\begin{equation*}
\geq \frac{1}{2}\frac{d}{dt}\left( f_{\delta }\left( \sqrt{\left\Vert \sqrt{%
\psi _{\varepsilon }}\nabla u\left( t\right) \right\Vert _{L^{2}\left(
\mathbb{R}
^{n}\right) }^{2}+\left\Vert \sqrt{\varphi _{\varepsilon }}\nabla u\left(
t\right) \right\Vert _{L^{2}\left(
\mathbb{R}
^{n}\right) }^{2}}\right) \left\Vert \eta _{r}\nabla \left( u\left( t\right)
\right) \right\Vert _{L^{2}\left(
\mathbb{R}
^{n}\right) }^{2}\right)
\end{equation*}%
\begin{equation}
-c_{4}\left( \left\Vert \sqrt{\beta }\nabla u_{t}\left( t\right) \right\Vert
_{L^{2}\left(
\mathbb{R}
^{n}\right) }+\left\Vert \sqrt{\alpha }u_{t}\left( t\right) \right\Vert
_{L^{2}\left(
\mathbb{R}
^{n}\right) }\right) \left\Vert \eta _{r}\nabla \left( u\left( t\right)
\right) \right\Vert _{L^{2}\left(
\mathbb{R}
^{n}\right) }^{2}.  \tag{2.35}
\end{equation}%
Considering (2.34) and (2.35) in (2.33), we obtain%
\begin{equation*}
\frac{1}{2}\frac{d}{dt}\left( \left\Vert \eta _{r}u_{t}\left( t\right)
\right\Vert _{L^{2}\left(
\mathbb{R}
^{n}\right) }^{2}+\gamma \left\Vert \eta _{r}\Delta u\left( t\right)
\right\Vert _{L^{2}\left(
\mathbb{R}
^{n}\right) }^{2}+\lambda \left\Vert \eta _{r}u\left( t\right) \right\Vert
_{L^{2}\left(
\mathbb{R}
^{n}\right) }^{2}\right)
\end{equation*}%
\begin{equation*}
+\frac{d}{dt}\left( \int\limits_{%
\mathbb{R}
^{n}}\eta _{r}^{2}\left( x\right) G\left( u\left( t,x\right) \right)
dx\right) +\left\Vert \sqrt{\beta }\eta _{r}\nabla u_{t}\left( t\right)
\right\Vert _{L^{2}\left(
\mathbb{R}
^{n}\right) }^{2}+\left\Vert \sqrt{\alpha }\eta _{r}u_{t}\left( t\right)
\right\Vert _{L^{2}\left(
\mathbb{R}
^{n}\right) }^{2}
\end{equation*}%
\begin{equation*}
+\frac{1}{2}\frac{d}{dt}\left( f_{\delta }\left( \sqrt{\left\Vert \sqrt{\psi
_{\varepsilon }}\nabla u\left( t\right) \right\Vert _{L^{2}\left(
\mathbb{R}
^{n}\right) }^{2}+\left\Vert \sqrt{\phi _{\varepsilon }}\nabla u\left(
t\right) \right\Vert _{L^{2}\left(
\mathbb{R}
^{n}\right) }^{2}}\right) \left\Vert \eta _{r}\nabla \left( u\left( t\right)
\right) \right\Vert _{L^{2}\left(
\mathbb{R}
^{n}\right) }^{2}\right)
\end{equation*}%
\begin{equation*}
-c_{4}\left( \left\Vert \sqrt{\beta }\nabla u_{t}\left( t\right) \right\Vert
_{L^{2}\left(
\mathbb{R}
^{n}\right) }+\left\Vert \sqrt{\alpha }u_{t}\left( t\right) \right\Vert
_{L^{2}\left(
\mathbb{R}
^{n}\right) }\right) \left\Vert \eta _{r}\nabla \left( u\left( t\right)
\right) \right\Vert _{L^{2}\left(
\mathbb{R}
^{n}\right) }^{2}.
\end{equation*}%
\begin{equation*}
-\varepsilon \left\Vert \eta _{r}\Delta u\left( t\right) \right\Vert
_{L^{2}\left(
\mathbb{R}
^{n}\right) }\left\Vert \eta _{r}u_{t}\left( t\right) \right\Vert
_{L^{2}\left(
\mathbb{R}
^{n}\right) }
\end{equation*}%
\begin{equation}
\leq c_{2}\left( \frac{1}{r}+\frac{1}{r}\left\Vert \sqrt{\beta }\eta
_{r}\nabla u_{t}\left( t\right) \right\Vert _{L^{2}\left( B\left(
0,2r\right) \right) }+\left\Vert h\right\Vert _{L^{2}\left(
\mathbb{R}
^{n}\backslash B\left( 0,r\right) \right) }\right) .  \tag{2.36}
\end{equation}%
Multiplying (1.1) with $\mu \eta _{r}^{2}u$, integrating the obtained
equality over $%
\mathbb{R}
^{n}$ and taking into account (1.6), (1.8) and (2.13), we get%
\begin{equation*}
\mu \gamma \left\Vert \eta _{r}\Delta u(t)\right\Vert _{L^{2}\left(
\mathbb{R}
^{n}\right) }^{2}+\mu \lambda \left\Vert \eta _{r}u(t)\right\Vert
_{L^{2}\left(
\mathbb{R}
^{n}\right) }^{2}-\mu \left\Vert \eta _{r}u_{t}\left( t\right) \right\Vert
_{L^{2}\left(
\mathbb{R}
^{n}\right) }^{2}
\end{equation*}%
\begin{equation*}
+\frac{\mu }{2}\frac{d}{dt}\left( \left\Vert \sqrt{\beta }\eta _{r}\nabla
u\left( t\right) \right\Vert _{L^{2}\left(
\mathbb{R}
^{n}\right) }^{2}+\left\Vert \sqrt{\alpha }\eta _{r}u\left( t\right)
\right\Vert _{L^{2}\left(
\mathbb{R}
^{n}\right) }^{2}\right)
\end{equation*}%
\begin{equation*}
+\mu \frac{d}{dt}\left( \int\limits_{%
\mathbb{R}
^{n}}\eta _{r}^{2}\left( x\right) u_{t}\left( t,x\right) u(t,x)dx\right)
\end{equation*}%
\begin{equation}
\leq c_{5}\left( \frac{1}{r}+\frac{1}{r}\left\Vert \sqrt{\beta }\eta
_{r}\nabla u_{t}\left( t\right) \right\Vert _{L^{2}\left( B\left(
0,2r\right) \right) }+\left\Vert h\right\Vert _{L^{2}\left(
\mathbb{R}
^{n}\backslash B\left( 0,r\right) \right) }\right) .  \tag{2.37}
\end{equation}%
Summing (2.36) and (2.37), applying Young inequality and choosing $%
\varepsilon $ and $\mu $ small enough, we obtain%
\begin{equation*}
\frac{d}{dt}\left( \left\Vert \eta _{r}u_{t}\left( t\right) \right\Vert
_{L^{2}\left(
\mathbb{R}
^{n}\right) }^{2}+\gamma \left\Vert \eta _{r}\Delta u\left( t\right)
\right\Vert _{L^{2}\left(
\mathbb{R}
^{n}\right) }^{2}+\lambda \left\Vert \eta _{r}u\left( t\right) \right\Vert
_{L^{2}\left(
\mathbb{R}
^{n}\right) }^{2}+\int\limits_{%
\mathbb{R}
^{n}\backslash B\left( 0,r\right) }\eta _{r}^{2}\left( x\right) G\left(
u\left( t,x\right) \right) dx\right)
\end{equation*}%
\begin{equation*}
+\frac{1}{2}\frac{d}{dt}\left( f_{\delta }\left( \sqrt{\left\Vert \sqrt{\psi
_{\varepsilon }}\nabla u\left( t\right) \right\Vert _{L^{2}\left(
\mathbb{R}
^{n}\right) }^{2}+\left\Vert \sqrt{\phi _{\varepsilon }}\nabla u\left(
t\right) \right\Vert _{L^{2}\left(
\mathbb{R}
^{n}\right) }^{2}}\right) \left\Vert \eta _{r}\nabla \left( u\left( t\right)
\right) \right\Vert _{L^{2}\left(
\mathbb{R}
^{n}\right) }^{2}\right)
\end{equation*}%
\begin{equation*}
+\frac{\mu }{2}\frac{d}{dt}\left( \left\Vert \sqrt{\beta }\eta _{r}\nabla
u\left( t\right) \right\Vert _{L^{2}\left(
\mathbb{R}
^{n}\right) }^{2}+\left\Vert \sqrt{\alpha }\eta _{r}u\left( t\right)
\right\Vert _{L^{2}\left(
\mathbb{R}
^{n}\right) }^{2}\right)
\end{equation*}%
\begin{equation*}
+c_{6}\left( \left\Vert \eta _{r}u_{t}\left( t\right) \right\Vert
_{L^{2}\left(
\mathbb{R}
^{n}\right) }^{2}+\gamma \left\Vert \eta _{r}\Delta u\left( t\right)
\right\Vert _{L^{2}\left(
\mathbb{R}
^{n}\right) }^{2}+\lambda \left\Vert \eta _{r}u\left( t\right) \right\Vert
_{L^{2}\left(
\mathbb{R}
^{n}\right) }^{2}\right)
\end{equation*}%
\begin{equation*}
\leq c_{7}\left( \left\Vert \sqrt{\beta }\nabla u_{t}\left( t\right)
\right\Vert _{L^{2}\left(
\mathbb{R}
^{n}\right) }+\left\Vert \sqrt{\alpha }u_{t}\left( t\right) \right\Vert
_{L^{2}\left(
\mathbb{R}
^{n}\right) }\right) \left\Vert \eta _{r}\nabla \left( u\left( t\right)
\right) \right\Vert _{L^{2}\left(
\mathbb{R}
^{n}\right) }^{2}
\end{equation*}%
\begin{equation*}
+c_{7}\left( \frac{1}{r}+\left\Vert h\right\Vert _{L^{2}\left(
\mathbb{R}
^{n}\backslash B\left( 0,r\right) \right) }\right) \text{, }\forall r\geq
r_{0}\text{,}
\end{equation*}%
where $c_{i}\left( i=6,7\right) $ are positive constants. By denoting
\begin{equation*}
\Phi \left( t\right) :=\left\Vert \eta _{r}u_{t}\left( t\right) \right\Vert
_{L^{2}\left(
\mathbb{R}
^{n}\right) }^{2}+\gamma \left\Vert \eta _{r}\Delta u\left( t\right)
\right\Vert _{L^{2}\left(
\mathbb{R}
^{n}\right) }^{2}+\lambda \left\Vert \eta _{r}u\left( t\right) \right\Vert
_{L^{2}\left(
\mathbb{R}
^{n}\right) }^{2}+\int\limits_{%
\mathbb{R}
^{n}}\eta _{r}^{2}\left( x\right) G\left( u\left( t,x\right) \right) dx
\end{equation*}%
\begin{equation*}
+\frac{1}{2}f_{\delta }\left( \sqrt{\left\Vert \sqrt{\psi _{\varepsilon }}%
\nabla u\left( t\right) \right\Vert _{L^{2}\left(
\mathbb{R}
^{n}\right) }^{2}+\left\Vert \sqrt{\phi _{\varepsilon }}\nabla u\left(
t\right) \right\Vert _{L^{2}\left(
\mathbb{R}
^{n}\right) }^{2}}\right) \left\Vert \eta _{r}\nabla \left( u\left( t\right)
\right) \right\Vert _{L^{2}\left(
\mathbb{R}
^{n}\right) }^{2}
\end{equation*}%
\begin{equation*}
+\frac{\mu }{2}\left( \left\Vert \sqrt{\beta }\eta _{r}\nabla u\left(
t\right) \right\Vert _{L^{2}\left(
\mathbb{R}
^{n}\backslash B\left( 0,r\right) \right) }^{2}+\left\Vert \sqrt{\alpha }%
\eta _{r}u\left( t\right) \right\Vert _{L^{2}\left(
\mathbb{R}
^{n}\backslash B\left( 0,r\right) \right) }^{2}\right) ,
\end{equation*}%
we get%
\begin{equation*}
\frac{d}{dt}\Phi \left( t\right) +c_{6}\left( \left\Vert \eta
_{r}u_{t}\left( t\right) \right\Vert _{L^{2}\left(
\mathbb{R}
^{n}\right) }^{2}+\gamma \left\Vert \eta _{r}\Delta u\left( t\right)
\right\Vert _{L^{2}\left(
\mathbb{R}
^{n}\right) }^{2}+\lambda \left\Vert \eta _{r}u\left( t\right) \right\Vert
_{L^{2}\left(
\mathbb{R}
^{n}\right) }^{2}\right)
\end{equation*}%
\begin{equation*}
\leq c_{7}\left( \left\Vert \sqrt{\beta }\nabla u_{t}\left( t\right)
\right\Vert _{L^{2}\left(
\mathbb{R}
^{n}\right) }+\left\Vert \sqrt{\alpha }u_{t}\left( t\right) \right\Vert
_{L^{2}\left(
\mathbb{R}
^{n}\right) }\right) \left\Vert \eta _{r}\nabla \left( u\left( t\right)
\right) \right\Vert _{L^{2}\left(
\mathbb{R}
^{n}\right) }^{2}
\end{equation*}%
\begin{equation}
+c_{7}\left( \frac{1}{r}+\left\Vert h\right\Vert _{L^{2}\left(
\mathbb{R}
^{n}\backslash B\left( 0,r\right) \right) }\right) \text{, }\forall r\geq
r_{0}\text{.}  \tag{2.38}
\end{equation}%
Moreover, there exist $\widehat{c}\equiv \widehat{c}\left( B\right) >0$ such
that%
\begin{equation*}
\left\Vert \eta _{r}u_{t}\left( t\right) \right\Vert _{L^{2}\left(
\mathbb{R}
^{n}\right) }^{2}+\gamma \left\Vert \eta _{r}\Delta u\left( t\right)
\right\Vert _{L^{2}\left(
\mathbb{R}
^{n}\right) }^{2}+\lambda \left\Vert \eta _{r}u\left( t\right) \right\Vert
_{L^{2}\left(
\mathbb{R}
^{n}\right) }^{2}
\end{equation*}%
\begin{equation}
\leq \Phi \left( t\right) \leq \widehat{c}\left( \left\Vert \eta
_{r}u_{t}\left( t\right) \right\Vert _{L^{2}\left(
\mathbb{R}
^{n}\right) }^{2}+\gamma \left\Vert \eta _{r}\Delta u\left( t\right)
\right\Vert _{L^{2}\left(
\mathbb{R}
^{n}\right) }^{2}+\lambda \left\Vert \eta _{r}u\left( t\right) \right\Vert
_{L^{2}\left(
\mathbb{R}
^{n}\right) }^{2}\right) .  \tag{2.39}
\end{equation}%
So, considering (2.39) in (2.38), we have%
\begin{equation*}
\frac{d}{dt}\Phi \left( t\right) +H\left( t\right) \Phi \left( t\right)
\end{equation*}%
\begin{equation*}
\leq c_{7}\left( \frac{1}{r}+\left\Vert h\right\Vert _{L^{2}\left(
\mathbb{R}
^{n}\backslash B\left( 0,r\right) \right) }\right) ,
\end{equation*}%
where $H\left( t\right) =c_{8}-c_{7}\left( \left\Vert \sqrt{\beta }\nabla
u_{t}\left( t\right) \right\Vert _{L^{2}\left(
\mathbb{R}
^{n}\right) }+\left\Vert \sqrt{\alpha }u_{t}\left( t\right) \right\Vert
_{L^{2}\left(
\mathbb{R}
^{n}\right) }\right) $ and $c_{8}>0$. Then, by Gronwall inequality, we obtain%
\begin{equation}
\Phi \left( t\right) \leq e^{-\int_{0}^{t}H\left( \tau \right) d\tau }\Phi
\left( 0\right) +c_{7}\left( \frac{1}{r}+\left\Vert h\right\Vert
_{L^{2}\left(
\mathbb{R}
^{n}\backslash B\left( 0,r\right) \right) }\right)
\int_{0}^{t}e^{-\int_{\tau }^{t}H\left( \sigma \right) d\sigma }d\tau .
\tag{2.40}
\end{equation}%
Furthermore, applying Young inequality and taking into account (1.9), we have%
\begin{equation*}
e^{-\int_{\tau }^{t}H\left( \sigma \right) d\sigma }\leq e^{-\frac{1}{2}%
c_{8}\left( t-\tau \right) +c_{9}\int_{\tau }^{t}\left( \left\Vert \sqrt{%
\beta }\nabla u_{t}\left( t\right) \right\Vert _{L^{2}\left(
\mathbb{R}
^{n}\right) }^{2}+\left\Vert \sqrt{\alpha }u_{t}\left( t\right) \right\Vert
_{L^{2}\left(
\mathbb{R}
^{n}\right) }^{2}\right) d\sigma \text{\ }}
\end{equation*}%
\begin{equation}
\leq c_{10}e^{-\frac{1}{2}c_{8}\left( t-\tau \right) }\text{, \ }\forall
t\geq \tau \geq 0\text{.}  \tag{2.41}
\end{equation}%
Therefore, considering (2.41) in (2.40), we get%
\begin{equation*}
\Phi \left( t\right) \leq c_{10}e^{-\frac{1}{2}c_{8}t}\Phi \left( 0\right)
+c_{11}\left( \frac{1}{r}+\left\Vert h\right\Vert _{L^{2}\left(
\mathbb{R}
^{n}\backslash B\left( 0,r\right) \right) }\right) \text{, \ }\forall t\geq 0%
\text{,}
\end{equation*}%
which completes the proof of the theorem.
\end{proof}

Now, we are in a position to prove the existence of the global attractor.

\begin{theorem}
Let the conditions (1.3)-(1.8) hold. Then the semigroup $\left\{ S\left(
t\right) \right\} _{t\geq 0}$ generated by the problem (1.1)-(1.2) possesses
a global attractor $\mathcal{A}$ in $H^{2}\left(
\mathbb{R}
^{n}\right) \times L^{2}\left(
\mathbb{R}
^{n}\right) $ and $\mathcal{A=M}^{u}\left( \mathcal{N}\right) $.
\end{theorem}

\begin{proof}
By Theorem 2.1 and Theorem 2.2, it follows that every sequence of the form $%
\left\{ S\left( t_{k}\right) \varphi _{k}\right\} _{k=1}^{\infty }$, where $%
\left\{ \varphi _{k}\right\} _{k=1}^{\infty }\subset \mathcal{B}$, $%
t_{k}\rightarrow \infty ,$ and $\mathcal{B}$ is bounded subset of $%
H^{2}\left(
\mathbb{R}
^{n}\right) \times L^{2}\left(
\mathbb{R}
^{n}\right) $,\ has a convergent subsequence in $H^{2}\left(
\mathbb{R}
^{n}\right) \times L^{2}\left(
\mathbb{R}
^{n}\right) $. Since, by (1.6) and (1.8), the set $\mathcal{N}$, which is
the set of stationary points of $\left\{ S\left( t\right) \right\} _{t\geq
0} $ is bounded in $H^{2}\left(
\mathbb{R}
^{n}\right) \times L^{2}\left(
\mathbb{R}
^{n}\right) $, to complete the proof, it is enough to show that the pair $%
\left( S\left( t\right) ,H^{2}\left(
\mathbb{R}
^{n}\right) \times L^{2}\left(
\mathbb{R}
^{n}\right) \right) $ is a gradient system (see \cite{14}).

Now, for $\left( u\left( t\right) ,u_{t}\left( t\right) \right) =S\left(
t\right) \left( u_{0},u_{1}\right) $, let the equality%
\begin{equation*}
L\left( u\left( t\right) ,u_{t}\left( t\right) \right) =L\left(
u_{0},u_{1}\right) \text{, \ }\forall t\geq 0\text{,}
\end{equation*}%
hold, where $L\left( u,v\right) =\frac{1}{2}\int\limits_{%
\mathbb{R}
^{n}}(\left\vert v\left( x\right) \right\vert ^{2}+\gamma \left\vert {\small %
\Delta u}\left( x\right) \right\vert ^{2}+\lambda \left\vert {\small u}%
\left( x\right) \right\vert ^{2})dx+\int\limits_{%
\mathbb{R}
^{n}}G\left( u\left( x\right) \right) dx+\frac{1}{2}F\left( \left\Vert
\nabla u\right\Vert _{L^{2}\left(
\mathbb{R}
^{n}\right) }^{2}\right) -\int\limits_{%
\mathbb{R}
^{n}}h\left( x\right) u(x)dx.$ Then considering (1.3) and (1.9), we have%
\begin{equation*}
\alpha u_{t}\left( t,\cdot \right) =0\text{ and }\beta \nabla u_{t}\left(
t,\cdot \right) =0\text{ \ a.e. in }%
\mathbb{R}
^{n},
\end{equation*}%
for $t\geq 0.$ Taking into account (1.5), from the above equalities, it
follows that%
\begin{equation*}
u_{t}\left( t,\cdot \right) u_{tx_{i}}\left( t,\cdot \right) =0\text{ a.e.
in }%
\mathbb{R}
^{n},
\end{equation*}%
and consequently%
\begin{equation*}
\frac{\partial }{\partial x_{i}}\left( u_{t}^{2}\left( t,\cdot \right)
\right) =0\text{ \ a.e. in }%
\mathbb{R}
^{n},
\end{equation*}%
for $i=\overline{1,n}$ and $t\geq 0.$ The last equality means that $%
u_{t}^{2}\left( t,\cdot \right) $ is independent of variable $x$, for every $%
t\geq 0.$ Hence, by $u_{t}\left( t,\cdot \right) \in L^{2}\left(
\mathbb{R}
^{n}\right) $, we have%
\begin{equation*}
u_{t}\left( t,\cdot \right) =0\text{ \ a.e. in }%
\mathbb{R}
^{n}\text{,}
\end{equation*}%
for $t\geq 0$. So,%
\begin{equation*}
\left( u\left( t\right) ,u_{t}\left( t\right) \right) =\left( \varphi
,0\right) \text{, \ \ }\forall t\geq 0,
\end{equation*}%
where $\left( \varphi ,0\right) \in \mathcal{N}$. Thus, the pair $\left(
S\left( t\right) ,H^{2}\left(
\mathbb{R}
^{n}\right) \times L^{2}\left(
\mathbb{R}
^{n}\right) \right) $ is a gradient system.
\end{proof}

\section{Regularity of the global attractor}

We start with the following lemma.

\begin{lemma}
Let the condition (1.7) hold and $K$ be a compact subset of $H^{2}\left(
\mathbb{R}
^{n}\right) $. Then for every $\varepsilon >0$ there exists a constant $%
C_{\epsilon }>0$ such that
\begin{equation}
\left\Vert g(u_{1})-g(u_{2})\right\Vert _{L^{2}\left(
\mathbb{R}
^{n}\right) }\leq \varepsilon \left\Vert u_{1}-u_{2}\right\Vert
_{H^{2}\left(
\mathbb{R}
^{n}\right) }+C_{\epsilon }\left\Vert u_{1}-u_{2}\right\Vert _{L^{2}\left(
\mathbb{R}
^{n}\right) },  \tag{3.1}
\end{equation}%
for every $u_{1},u_{2}\in K$.
\end{lemma}

\begin{proof}
By Mean Value Theorem, H\"{o}lder inequality and\ the embedding $H^{2}\left(
\mathbb{R}
^{n}\right) \hookrightarrow L^{\frac{2n}{\left( n-4\right) ^{+}}}\left(
\mathbb{R}
^{n}\right) \cap L^{2}\left(
\mathbb{R}
^{n}\right) $, we have%
\begin{equation*}
\left\Vert g\left( u\right) -g\left( v\right) \right\Vert _{L^{2}\left(
\mathbb{R}
^{n}\right) }^{2}=\int\limits_{%
\mathbb{R}
^{n}}\left\vert \int\limits_{0}^{1}g^{\prime }\left( \tau u\left( x\right)
+\left( 1-\tau \right) v\left( x\right) \right) d\tau \right\vert
^{2}\left\vert u\left( x\right) -v\left( x\right) \right\vert ^{2}dx
\end{equation*}%
\begin{equation*}
\leq \int\limits_{0}^{1}\int\limits_{\left\{ x\in
\mathbb{R}
^{n}:\left\vert \tau u\left( x\right) +\left( 1-\tau \right) v\left(
x\right) \right\vert >M\right\} }\left\vert g^{\prime }\left( \tau u\left(
x\right) +\left( 1-\tau \right) v\left( x\right) \right) \right\vert
^{2}\left\vert u\left( x\right) -v\left( x\right) \right\vert ^{2}dxd\tau
\end{equation*}%
\begin{equation*}
+\left\Vert g^{\prime }\right\Vert _{C\left[ -M,M\right] }\left\Vert
u-v\right\Vert _{L^{2}\left(
\mathbb{R}
^{n}\right) }^{2}
\end{equation*}%
\begin{equation*}
\leq c_{1}\int\limits_{\left\{ x\in
\mathbb{R}
^{n}:\left\vert u\left( x\right) \right\vert +\left\vert v\left( x\right)
\right\vert >M\right\} }\left( 1+\left\vert u\left( x\right) \right\vert
^{2\left( p-1\right) }+\left\vert v\left( x\right) \right\vert ^{2\left(
p-1\right) }\right) \left\vert u\left( x\right) -v\left( x\right)
\right\vert ^{2}dx
\end{equation*}%
\begin{equation*}
+\left\Vert g^{\prime }\right\Vert _{C\left[ -M,M\right] }\left\Vert
u-v\right\Vert _{L^{2}\left(
\mathbb{R}
^{n}\right) }^{2}
\end{equation*}%
\begin{equation*}
\leq c_{2}\left( \int\limits_{\left\{ x\in
\mathbb{R}
^{n}:\left\vert u\left( x\right) \right\vert +\left\vert v\left( x\right)
\right\vert >M\right\} }\left( 1+\left\vert u\left( x\right) \right\vert
^{2\left( p-1\right) q}+\left\vert v\left( x\right) \right\vert ^{2\left(
p-1\right) q}\right) dx\right) ^{\frac{1}{q}}
\end{equation*}%
\begin{equation}
\times \left\Vert u-v\right\Vert _{H^{2}\left(
\mathbb{R}
^{n}\right) }+\left\Vert g^{\prime }\right\Vert _{C\left[ -M,M\right]
}\left\Vert u-v\right\Vert _{L^{2}\left(
\mathbb{R}
^{n}\right) }^{2},  \tag{3.2}
\end{equation}%
where $q=$ $\max \left\{ 1,\frac{n}{4}\right\} $.

Since, by (1.7), $H^{2}\left(
\mathbb{R}
^{n}\right) \hookrightarrow L^{2\left( p-1\right) q}\left(
\mathbb{R}
^{n}\right) $, we have that $K$ is compact subset of $L^{2\left( p-1\right)
q}\left(
\mathbb{R}
^{n}\right) .$ Hence,%
\begin{equation}
\lim_{M\rightarrow \infty }\sup_{u,v\in K}\int\limits_{\left\{ x\in
\mathbb{R}
^{n}:\left\vert u\left( x\right) \right\vert +\left\vert v\left( x\right)
\right\vert >M\right\} }\left( 1+\left\vert u\left( x\right) \right\vert
^{2\left( p-1\right) q}+\left\vert v\left( x\right) \right\vert ^{2\left(
p-1\right) q}\right) dx=0.  \tag{3.3}
\end{equation}%
Thus, (3.2) and (3.3) give us (3.1).
\end{proof}

\begin{theorem}
The global attractor $\mathcal{A}$ is bounded in $H^{3}\left(
\mathbb{R}
^{n}\right) \times H^{2}\left(
\mathbb{R}
^{n}\right) $.
\end{theorem}

\begin{proof}
Let $\varphi \in \mathcal{A}$. Since $\mathcal{A}$ is invariant, there
exists an invariant trajectory $\Gamma =\left\{ \left( u\left( t\right)
,u_{t}\left( t\right) \right) :t\in
\mathbb{R}
\right\} $\newline
$\subset \mathcal{A}$ such that $\left( u\left( 0\right) ,u_{t}\left(
0\right) \right) =\varphi $ (see \cite[p. 159]{18}). Now, let us define
\begin{equation*}
v\left( t,x\right) :=\frac{u\left( t+\sigma ,x\right) -u\left( t,x\right) }{%
\sigma }\text{, \ }\sigma >0\text{.}
\end{equation*}%
Then, by (1.1), we get%
\begin{equation*}
v_{tt}(t,x)+\gamma \Delta ^{2}v(t,x)-{div}\left( \beta \left( x\right)
\nabla v_{t}\right) +\alpha (x)v_{t}(t,x)+\lambda v(t,x)
\end{equation*}%
\begin{equation*}
-f\left( \left\Vert \nabla u\left( t\right) \right\Vert _{L^{2}\left(
\mathbb{R}
^{n}\right) }\right) \Delta v\left( t,x\right) -\frac{f(\left\Vert \nabla
u\left( t+\sigma \right) \right\Vert _{L^{2}\left(
\mathbb{R}
^{n}\right) })-f\left( \left\Vert \nabla u\left( t\right) \right\Vert
_{L^{2}\left(
\mathbb{R}
^{n}\right) }\right) }{\sigma }\Delta u(t+\sigma ,x)
\end{equation*}%
\begin{equation}
+\frac{g\left( u\left( t+\sigma ,x\right) \right) -g\left( u\left(
t,x\right) \right) }{\sigma }=0\text{, \ \ \ }(t,x)\in
\mathbb{R}
\times
\mathbb{R}
^{n}\text{.}  \tag{3.4}
\end{equation}%
Multiplying (3.4) by $v_{t}$ and integrating the obtained equality over $%
\mathbb{R}
^{n}$, we find%
\begin{equation*}
\frac{d}{dt}E(v(t))+\left\Vert \sqrt{\beta }\nabla v_{t}\left( t\right)
\right\Vert _{L^{2}\left(
\mathbb{R}
^{n}\right) }^{2}+\left\Vert \sqrt{\alpha }v_{t}\left( t\right) \right\Vert
_{L^{2}\left(
\mathbb{R}
^{n}\right) }^{2}
\end{equation*}%
\begin{equation*}
\leq -\frac{1}{2}f\left( \left\Vert \nabla u\left( t\right) \right\Vert
_{L^{2}\left(
\mathbb{R}
^{n}\right) }\right) \frac{d}{dt}\left( \left\Vert \nabla v\left( t\right)
\right\Vert _{L^{2}\left(
\mathbb{R}
^{n}\right) }^{2}\right) +\frac{f(\left\Vert \nabla u\left( t+\sigma \right)
\right\Vert _{L^{2}\left(
\mathbb{R}
^{n}\right) })-f\left( \left\Vert \nabla u\left( t\right) \right\Vert
_{L^{2}\left(
\mathbb{R}
^{n}\right) }\right) }{\sigma }
\end{equation*}%
\begin{equation*}
\times \int\limits_{%
\mathbb{R}
^{n}}\Delta u(t+\sigma ,x)v_{t}\left( t,x\right) dx-\frac{1}{\sigma }%
\int\limits_{%
\mathbb{R}
^{n}}\left( g\left( u\left( t+\sigma ,x\right) \right) -g\left( u\left(
t,x\right) \right) \right) v_{t}\left( t,x\right) dx
\end{equation*}%
\begin{equation*}
\leq -\frac{1}{2}f\left( \left\Vert \nabla u\left( t\right) \right\Vert
_{L^{2}\left(
\mathbb{R}
^{n}\right) }\right) \frac{d}{dt}\left( \left\Vert \nabla v\left( t\right)
\right\Vert _{L^{2}\left(
\mathbb{R}
^{n}\right) }^{2}\right)
\end{equation*}%
\begin{equation*}
+c_{1}\left\Vert \nabla v\left( t\right) \right\Vert _{L^{2}\left(
\mathbb{R}
^{n}\right) }\left\Vert v_{t}\left( t\right) \right\Vert _{L^{2}\left(
\mathbb{R}
^{n}\right) }+\frac{1}{\sigma }\left\Vert g\left( u\left( t+\sigma \right)
\right) -g\left( u\left( t\right) \right) \right\Vert _{L^{2}\left(
\mathbb{R}
^{n}\right) }\left\Vert v_{t}\left( t\right) \right\Vert _{L^{2}\left(
\mathbb{R}
^{n}\right) }.
\end{equation*}%
Taking into account Lemma 3.1 in the last inequality, we obtain%
\begin{equation*}
\frac{d}{dt}E(v(t))+\left\Vert \sqrt{\beta }\nabla v_{t}\left( t\right)
\right\Vert _{L^{2}\left(
\mathbb{R}
^{n}\right) }^{2}+\left\Vert \sqrt{\alpha }v_{t}\left( t\right) \right\Vert
_{L^{2}\left(
\mathbb{R}
^{n}\right) }^{2}
\end{equation*}%
\begin{equation*}
\leq -\frac{1}{2}f\left( \left\Vert \nabla u\left( t\right) \right\Vert
_{L^{2}\left(
\mathbb{R}
^{n}\right) }\right) \frac{d}{dt}\left( \left\Vert \nabla v\left( t\right)
\right\Vert _{L^{2}\left(
\mathbb{R}
^{n}\right) }^{2}\right)
\end{equation*}%
\begin{equation}
+c_{1}\left\Vert \nabla v\left( t\right) \right\Vert _{L^{2}\left(
\mathbb{R}
^{n}\right) }\left\Vert v_{t}\left( t\right) \right\Vert _{L^{2}\left(
\mathbb{R}
^{n}\right) }+\left( \varepsilon \left\Vert v\left( t\right) \right\Vert
_{H^{2}\left(
\mathbb{R}
^{n}\right) }+C_{\varepsilon }\left\Vert v\left( t\right) \right\Vert
_{L^{2}\left(
\mathbb{R}
^{n}\right) }\right) \left\Vert v_{t}\left( t\right) \right\Vert
_{L^{2}\left(
\mathbb{R}
^{n}\right) },  \tag{3.5}
\end{equation}%
for any $\varepsilon >0$. Moreover, by (2.13), we have
\begin{equation}
\left\Vert v\left( t\right) \right\Vert _{L^{2}\left(
\mathbb{R}
^{n}\right) }=\left\Vert \frac{u\left( t+\sigma ,x\right) -u\left(
t,x\right) }{\sigma }\right\Vert _{L^{2}\left(
\mathbb{R}
^{n}\right) }\leq \sup_{0\leq t<\infty }\left\Vert u_{t}\left( t\right)
\right\Vert _{L^{2}\left(
\mathbb{R}
^{n}\right) }<\widehat{C}\text{, \ }\forall t\in
\mathbb{R}
\text{.}  \tag{3.6}
\end{equation}%
Then, considering (3.6) in (3.5), we get%
\begin{equation*}
\frac{d}{dt}E(v(t))+\left\Vert \sqrt{\beta }\nabla v_{t}\left( t\right)
\right\Vert _{L^{2}\left(
\mathbb{R}
^{n}\right) }^{2}+\left\Vert \sqrt{\alpha }v_{t}\left( t\right) \right\Vert
_{L^{2}\left(
\mathbb{R}
^{n}\right) }^{2}
\end{equation*}%
\begin{equation*}
\leq -\frac{1}{2}f\left( \left\Vert \nabla u\left( t\right) \right\Vert
_{L^{2}\left(
\mathbb{R}
^{n}\right) }\right) \frac{d}{dt}\left( \left\Vert \nabla v\left( t\right)
\right\Vert _{L^{2}\left(
\mathbb{R}
^{n}\right) }^{2}\right)
\end{equation*}%
\begin{equation}
+\left( c_{2}\left\Vert v\left( t\right) \right\Vert _{H^{2}\left(
\mathbb{R}
^{n}\right) }^{\frac{1}{2}}+\varepsilon \left\Vert v\left( t\right)
\right\Vert _{H^{2}\left(
\mathbb{R}
^{n}\right) }+\widetilde{C_{\varepsilon }}\right) \left\Vert v_{t}\left(
t\right) \right\Vert _{L^{2}\left(
\mathbb{R}
^{n}\right) }.  \tag{3.7}
\end{equation}%
Now, let us estimate the first term on the right hand side of (3.7). By
(2.13) and (3.6), we have%
\begin{equation*}
-\frac{1}{2}f\left( \left\Vert \nabla u\left( t\right) \right\Vert
_{L^{2}\left(
\mathbb{R}
^{n}\right) }\right) \frac{d}{dt}\left( \left\Vert \nabla v\left( t\right)
\right\Vert _{L^{2}\left(
\mathbb{R}
^{n}\right) }^{2}\right)
\end{equation*}%
\begin{equation*}
\leq c_{3}\max_{0\leq s_{1},s_{2}\leq \varepsilon }\left\vert f\left(
s_{1}\right) -f\left( s_{2}\right) \right\vert \left\Vert v_{t}\left(
t\right) \right\Vert _{L^{2}\left(
\mathbb{R}
^{n}\right) }\left\Vert \Delta v\left( t\right) \right\Vert _{L^{2}\left(
\mathbb{R}
^{n}\right) }
\end{equation*}%
\begin{equation*}
-\frac{1}{2}f_{\varepsilon }\left( \left\Vert \nabla u\left( t\right)
\right\Vert _{L^{2}\left(
\mathbb{R}
^{n}\right) }\right) \frac{d}{dt}\left( \left\Vert \nabla v\left( t\right)
\right\Vert _{L^{2}\left(
\mathbb{R}
^{n}\right) }^{2}\right)
\end{equation*}%
\begin{equation*}
\leq c_{3}\max_{0\leq s_{1},s_{2}\leq \varepsilon }\left\vert f\left(
s_{1}\right) -f\left( s_{2}\right) \right\vert \left\Vert v_{t}\left(
t\right) \right\Vert _{L^{2}\left(
\mathbb{R}
^{n}\right) }\left\Vert \Delta v\left( t\right) \right\Vert _{L^{2}\left(
\mathbb{R}
^{n}\right) }
\end{equation*}%
\begin{equation}
-\frac{1}{2}\frac{d}{dt}\left( f_{\varepsilon }\left( \left\Vert \nabla
u\left( t\right) \right\Vert _{L^{2}\left(
\mathbb{R}
^{n}\right) }\right) \left\Vert \nabla v\left( t\right) \right\Vert
_{L^{2}\left(
\mathbb{R}
^{n}\right) }^{2}\right) +c_{4}\left\Vert v\left( t\right) \right\Vert
_{H^{2}\left(
\mathbb{R}
^{n}\right) },  \tag{3.8}
\end{equation}%
for any $\varepsilon >0$, where $f_{\varepsilon }$ is the function defined
in the proof of Lemma 2.1. Considering (3.8) in (3.7), we obtain%
\begin{equation*}
\frac{d}{dt}\left( E(v(t))+\frac{1}{2}f_{\varepsilon }\left( \left\Vert
\nabla u\left( t\right) \right\Vert _{L^{2}\left(
\mathbb{R}
^{n}\right) }\right) \left\Vert \nabla v\left( t\right) \right\Vert
_{L^{2}\left(
\mathbb{R}
^{n}\right) }^{2}\right) +\left\Vert \sqrt{\beta }\nabla v_{t}\left(
t\right) \right\Vert _{L^{2}\left(
\mathbb{R}
^{n}\right) }^{2}+\left\Vert \sqrt{\alpha }v_{t}\left( t\right) \right\Vert
_{L^{2}\left(
\mathbb{R}
^{n}\right) }^{2}
\end{equation*}%
\begin{equation*}
\leq c_{3}\max_{0\leq s_{1},s_{2}\leq \varepsilon }\left\vert f\left(
s_{1}\right) -f\left( s_{2}\right) \right\vert \left\Vert v_{t}\left(
t\right) \right\Vert _{L^{2}\left(
\mathbb{R}
^{n}\right) }\left\Vert \Delta u\left( t\right) \right\Vert _{L^{2}\left(
\mathbb{R}
^{n}\right) }+c_{4}\left\Vert v\left( t\right) \right\Vert _{H^{2}\left(
\mathbb{R}
^{n}\right) }
\end{equation*}%
\begin{equation}
+\left( c_{2}\left\Vert v\left( t\right) \right\Vert _{H^{2}\left(
\mathbb{R}
^{n}\right) }^{\frac{1}{2}}+\varepsilon \left\Vert v\left( t\right)
\right\Vert _{H^{2}\left(
\mathbb{R}
^{n}\right) }+\widetilde{C_{\varepsilon }}\right) \left\Vert v_{t}\left(
t\right) \right\Vert _{L^{2}\left(
\mathbb{R}
^{n}\right) }.  \tag{3.9}
\end{equation}

Let $\eta _{r}\left( x\right) $ be the cut-off function defined in the proof
of Theorem 2.1. Multiplying (3.4) by\newline
$\sum\nolimits_{i=1}^{n}x_{i}\left( 1-\eta _{2r_{0}}\right) v_{x_{i}}+\frac{1%
}{2}\left( n-1\right) \left( 1-\eta _{2r_{0}}\right) v$, and integrating
over $%
\mathbb{R}
^{n}$, by (2.13) and (3.6), we get%
\begin{equation*}
\frac{3}{2}\gamma \left\Vert \Delta \left( v\left( t\right) \right)
\right\Vert _{L^{2}\left( B\left( 0,2r_{0}\right) \right) }^{2}+\frac{1}{2}%
\left\Vert v_{t}\left( t\right) \right\Vert _{L^{2}\left( B\left(
0,2r_{0}\right) \right) }^{2}
\end{equation*}%
\begin{equation*}
+\frac{d}{dt}\left( \sum\nolimits_{i=1}^{n}\int\limits_{%
\mathbb{R}
^{n}}x_{i}\left( 1-\eta _{2r_{0}}\left( x\right) \right) v_{x_{i}}\left(
t,x\right) v_{t}\left( t,x\right) dx+\frac{1}{2}\left( n-1\right)
\int\limits_{%
\mathbb{R}
^{n}}\left( 1-\eta _{2r_{0}}\left( x\right) \right) v_{t}\left( t,x\right)
v\left( t,x\right) dx\right)
\end{equation*}%
\begin{equation*}
\leq c_{5}\left\Vert v_{t}\left( t\right) \right\Vert _{L^{2}\left( B\left(
0,4r_{0}\right) \backslash B\left( 0,2r_{0}\right) \right)
}^{2}+c_{5}\left\Vert \Delta v\left( t\right) \right\Vert _{L^{2}\left(
B\left( 0,4r_{0}\right) \backslash B\left( 0,2r_{0}\right) \right)
}^{2}+c_{5}\left\Vert v\left( t\right) \right\Vert _{H^{2}(B\left(
0,4r_{0}\right) )}^{\frac{1}{2}}
\end{equation*}%
\begin{equation*}
+c_{5}\left\Vert \sqrt{\beta }\nabla v_{t}\left( t\right) \right\Vert
_{L^{2}(B\left( 0,4r_{0}\right) )}\left( \left\Vert v\left( t\right)
\right\Vert _{H^{2}(B\left( 0,4r_{0}\right) )}^{\frac{1}{2}}+\left\Vert
v\left( t\right) \right\Vert _{H^{2}(B\left( 0,4r_{0}\right) )}\right)
\end{equation*}%
\begin{equation*}
+c_{5}\left\Vert \sqrt{\alpha }v_{t}\left( t\right) \right\Vert
_{L^{2}(B\left( 0,4r_{0}\right) )}\left( \left\Vert v\left( t\right)
\right\Vert _{H^{2}(B\left( 0,4r_{0}\right) )}^{\frac{1}{2}}+1\right)
\end{equation*}%
\begin{equation}
+c_{5}\left\Vert v\left( t\right) \right\Vert _{H^{2}(B\left(
0,4r_{0}\right) )}^{\frac{3}{2}}+c_{5}.  \tag{3.10}
\end{equation}%
Multiplying (3.4) by $\eta _{r_{0}}^{2}v$ and integrating over $%
\mathbb{R}
^{n}$, we find%
\begin{equation*}
\frac{d}{dt}\left( \int\limits_{%
\mathbb{R}
^{n}}\eta _{r_{0}}^{2}v\left( t,x\right) v_{t}\left( t,x\right) dx+\frac{1}{2%
}\left\Vert \sqrt{\alpha }\eta _{r_{0}}v\left( t\right) \right\Vert
_{L^{2}\left(
\mathbb{R}
^{n}\right) }^{2}+\frac{1}{2}\left\Vert \sqrt{\beta }\eta _{r_{0}}\nabla
v\left( t\right) \right\Vert _{L^{2}\left(
\mathbb{R}
^{n}\right) }^{2}\right)
\end{equation*}%
\begin{equation*}
+\left\Vert \Delta v\left( t\right) \right\Vert _{L^{2}\left(
\mathbb{R}
^{n}\backslash B\left( 0,r_{0}\right) \right) }^{2}+\lambda \left\Vert
v\left( t\right) \right\Vert _{L^{2}\left(
\mathbb{R}
^{n}\backslash B\left( 0,r_{0}\right) \right) }^{2}-\left\Vert v_{t}\left(
t\right) \right\Vert _{L^{2}\left(
\mathbb{R}
^{n}\backslash B\left( 0,r_{0}\right) \right) }^{2}
\end{equation*}%
\begin{equation}
\leq c_{6}\left\Vert v\left( t\right) \right\Vert _{H^{2}\left(
\mathbb{R}
^{n}\right) }^{\frac{3}{2}}+c_{6}\left\Vert \sqrt{\beta }\nabla v_{t}\left(
t\right) \right\Vert _{L^{2}\left(
\mathbb{R}
^{n}\right) }\left( 1+\left\Vert v\left( t\right) \right\Vert
_{H^{2}(B\left( 0,4r_{0}\right) )}^{\frac{1}{2}}\right) +c_{6}.  \tag{3.11}
\end{equation}%
Multiplying (3.10) and (3.11) by $\delta ^{2}$ and $\delta $, respectively,
then summing the obtained inequalities with (3.9), choosing $\varepsilon >0$
and $\delta >0$ sufficiently small and applying Young inequality, we get%
\begin{equation}
\frac{d}{dt}\Psi \left( t\right) +c_{7}E\left( v\left( t\right) \right) \leq
c_{8}\text{, \ }\forall t\in
\mathbb{R}
\text{,}  \tag{3.12}
\end{equation}%
where%
\begin{equation*}
\Psi \left( t\right) :=E(v(t))+\frac{1}{2}f_{\varepsilon }\left( \left\Vert
\nabla u\left( t\right) \right\Vert _{L^{2}\left(
\mathbb{R}
^{n}\right) }\right) \left\Vert \nabla v\left( t\right) \right\Vert
_{L^{2}\left(
\mathbb{R}
^{n}\right) }^{2}
\end{equation*}%
\begin{equation*}
+\delta \left( \int\limits_{%
\mathbb{R}
^{n}}\eta _{r_{0}}^{2}v\left( t,x\right) v_{t}\left( t,x\right) dx+\frac{1}{2%
}\left\Vert \sqrt{\alpha }\eta _{r_{0}}v\left( t\right) \right\Vert
_{L^{2}\left(
\mathbb{R}
^{n}\right) }^{2}+\frac{1}{2}\left\Vert \sqrt{\beta }\eta _{r_{0}}\nabla
v\left( t\right) \right\Vert _{L^{2}\left(
\mathbb{R}
^{n}\right) }^{2}\right)
\end{equation*}%
\begin{equation*}
+\delta ^{2}\left( \sum\nolimits_{i=1}^{n}\int\limits_{%
\mathbb{R}
^{n}}x_{i}\left( 1-\eta _{2r_{0}}\left( x\right) \right) v_{x_{i}}\left(
t,x\right) v_{t}\left( t,x\right) dx+\frac{1}{2}\left( n-1\right)
\int\limits_{%
\mathbb{R}
^{n}}\left( 1-\eta _{2r_{0}}\left( x\right) \right) v_{t}\left( t,x\right)
v\left( t,x\right) dx\right) ,
\end{equation*}%
and the positive constant $c_{8}$, as the previous $c_{i}$ $\left( i=%
\overline{1,7}\right) $, is independent of the trajectory $\Gamma .$

Since $\delta >0$ is sufficiently small, there exist constants $c>0,$ $%
\widetilde{c}>0$ such that%
\begin{equation}
cE\left( v\left( t\right) \right) \leq \Psi \left( t\right) \leq \widetilde{c%
}E\left( v\left( t\right) \right) \text{, \ \ }\forall t\in
\mathbb{R}
\text{.}  \tag{3.13}
\end{equation}%
Taking into account (3.13) in (3.12), we obtain%
\begin{equation*}
\frac{d}{dt}\Psi \left( t\right) +c_{9}\Psi \left( t\right) \leq c_{8}\text{%
, \ \ }\forall t\in
\mathbb{R}
\text{,}
\end{equation*}%
which yields%
\begin{equation*}
\Psi \left( t\right) \leq e^{-c_{9}(t-s)}\Psi \left( s\right) +\frac{c_{8}}{%
c_{9}}\text{, \ \ }\forall t\geq s\text{.}
\end{equation*}%
Passing to the limit as $s\rightarrow -\infty $ and considering (3.13), we
get%
\begin{equation*}
E\left( v\left( t\right) \right) \leq c_{10}\text{, \ \ }\forall t\in
\mathbb{R}
\text{.}
\end{equation*}%
By using the definition of $v$, after passing to the limit as $\sigma
\rightarrow 0$ in the last inequality, we find%
\begin{equation}
E\left( u_{t}\left( t\right) \right) \leq c_{10}\text{, \ }\forall t\in
\mathbb{R}
\text{.}  \tag{3.14}
\end{equation}%
Considering (3.14) in (1.1), we obtain%
\begin{equation*}
\left\Vert u\left( t\right) \right\Vert _{H^{3}\left(
\mathbb{R}
^{n}\right) }\leq c_{11}\text{, \ \ }\forall t\in
\mathbb{R}
\text{.}
\end{equation*}%
Thus, the last inequality, together with (3.14), yields%
\begin{equation*}
\left\Vert \varphi \right\Vert _{H^{3}\left(
\mathbb{R}
^{n}\right) \times H^{2}\left(
\mathbb{R}
^{n}\right) }\leq c_{12}\text{, \ \ }\forall \varphi \in \mathcal{A}\text{,}
\end{equation*}%
which completes the proof of the theorem.
\end{proof}

\end{document}